\documentclass[a4paper,11pt]{article}

\usepackage{geometry} 
\usepackage{amsmath}
\usepackage{graphicx}
\usepackage{amssymb}
\usepackage{epstopdf}
\usepackage{float}
\usepackage{caption}
\usepackage{hyperref}
\usepackage{mathtools}
\usepackage{color}
\usepackage{listings}
\usepackage{nicefrac}

\usepackage[style=alphabetic, isbn=false, doi=false, maxnames=10, backend=biber]{biblatex}
\renewbibmacro{in:}{}
\DeclareFieldFormat
  [article,inbook,incollection,inproceedings,patent,thesis,unpublished]
  {title}{#1\isdot}
\usepackage[bblfile=contact_process_on_adaptive]{biblatex-readbbl}

\usepackage[normalem]{ulem}
\usepackage{ mathrsfs }
\usepackage{bbm}
\usepackage[dvipsnames]{xcolor}
\usepackage{todonotes}
\usepackage[shortlabels]{enumitem}
\usepackage{amsthm}

\usepackage{floatpag}

\newtheorem{theorem}{Theorem}[section]

\newtheorem{proposition}[theorem]{Proposition}

\newtheorem{lemma}[theorem]{Lemma}
\newtheorem*{theorem*}{Theorem}

\theoremstyle{definition}
\newtheorem{definition}[theorem]{Definition}
\newtheorem{remark}[theorem]{Remark}

\numberwithin{equation}{section}

\usepackage{ wasysym }
\usepackage{ textcomp }
\usepackage[most]{tcolorbox}
\usepackage{tikz}

\newcommand*\mycirc[1]{
  \begin{tikzpicture}
    \node[draw,circle,inner sep=1pt] {#1};
  \end{tikzpicture}}

\newcommand{\1}{\mathbbm{1}}

\newcommand{\N}{\mathbb{N}}
\newcommand{\R}{\mathbb{R}}

\newcommand{\cT}{\mathcal{T}}

\newcommand{\de}{\operatorname{d}}
\newcommand{\ra}{\rightarrow}

\newcommand{\sT}{\mathscr{T}}
\newcommand{\sF}{\mathscr{F}}

\DeclareMathOperator{\E}{\mathbb{E}}
\DeclareMathOperator{\p}{\mathbb{P}}

\newcommand{\la}{\lambda}
\renewcommand{\epsilon}{\varepsilon}
\newcommand{\eps}{\epsilon}
 \newcommand{\ssup}[1] {{\scriptscriptstyle{({#1}})}}

\title{The Contact Process on a Graph Adapting to the Infection}
\author{John Fernley, Peter M\"orters, Marcel Ortgiese}
\date{\today}

\usepackage{framed}
\usepackage{float}
\usepackage[caption = false]{subfig}
\usepackage{graphicx}
\usepackage{makecell}

\usepackage{appendix}
\usepackage{appendix}

\DeclareRobustCommand{\VAN}[3]{#2}

\begin{document}

\begin{center}
{\LARGE  The Contact Process on a Graph Adapting to the Infection}

\vspace{1.5em}

John Fernley\footnote{HUN-REN R\'enyi Alfr\'ed Matematikai Kutat\'oint\'ezet,
			Re\'altanoda utca 13-15,
            1053 Budapest,
            Hungary. {\tt fernley@renyi.hu}}\qquad
Peter M\"orters\footnote{Mathematisches Institut,
            Universit\"at zu K\"oln,
            Weyertal 86-90,
            50931 K\"oln,
            Germany. {\tt moerters@math.uni-koeln.de}}\qquad
Marcel Ortgiese\footnote{Department of Mathematical Sciences,
            University of Bath,
            Claverton Down,
            Bath,
            BA2 7AY,
            United Kingdom. {\tt m.ortgiese@bath.ac.uk}}

\vspace{1.5em}

\today
\end{center}

\begin{abstract}
We show existence of  a non-trivial phase transition for the contact process, a simple model for infection without immunity, on a network  which reacts dynamically to the infection trying to prevent an epidemic.
This network initially has the distribution of an Erd\H{o}s-R\'enyi graph, but is made adaptive via updating in only the infected neighbourhoods, at constant rate. Under these graph dynamics, the presence of infection can help to prevent the spread and so many monotonicity-based techniques fail. Instead, we approximate the infection by an idealised model, the contact process on an evolving forest (CPEF). In this model updating events of a vertex in a tree correspond to the creation of a new tree with that vertex as the root. By seeing the newly created trees as offspring, the CPEF generates a Galton-Watson tree and we use that it survives if and only if its offspring number exceeds one.  
It turns out that the phase transition in the adaptive model occurs at a different infection rate when compared to the non-adaptive model.

{\footnotesize
\vspace{1em}

}
\end{abstract}

\section{Introduction}\label{section_contact_process_summary}

The contact process is a simple Markovian model of an infection spreading on a graph. 
Every reasonable infection model must account for both transmission and recovery and, if we are interested in Markov models, each must happen at constant rate. By moving to a unitless relative measure of time we can assume, as is conventional, that recoveries for each vertex occur independently at rate $1$ and that transmissions of the infection occur for each edge independently at \emph{relative} rate $\lambda>0$.
Thus we can define the process by saying each infected vertex infects each of its neighbours  at rate $\lambda$, while recovering at rate $1$.
This means that the rate of infection of each healthy vertex increases linearly with its number of infected neighbours. After every infection we return to susceptibility, and so the contact process will best model fast-mutating infections with low mortality.\smallskip

The contact process satisfies a  self-duality \cite{liggett1999stochastic}, which is useful in many proof techniques. 
It was first seen on the lattice~\cite{harris1974contact}, but to understand infections with a realistic differentiation of people we should instead look at the random graphs which frequently model social networks. Even better, to understand the effect of temporal variability in human interaction on the infection process which these interaction networks carry, we should look at random graphs which are dynamic in time -- including the most challenging cases where these dynamics involve responding to the infection.\smallskip

Naturally, the contact process has not been studied on dynamic random graphs as much as on static random graphs. The most prominent static random model is the configuration model, which is  uniformly distributed on all graphs with a given number $N$ of vertices and degree sequence of length $N$. On this graph \cite{bhamidi2019survival} establish that degree sequences drawn independently from a distribution with a finite exponential moment have $\lambda_c>0$ such that infection rates $\lambda \in (0, \lambda_c)$ cannot spread for more than $N^{1+o(1)}$ time, with high probability -- note that this includes the Erd\H{o}s-R\'enyi graph. They also show that any distribution without a finite exponential moment has epidemics surviving with high probability for $e^{\Theta(N)}$ time whenever $\lambda>0$, which expands on the result of \cite{chatterjee2009contact} for power law distributions. The article \cite{chatterjee2009contact} was the first correction of \cite{pastor2001epidemic}, the latter claiming based on heuristics that power law distributions with finite second moment can have fast extinction. This controversy reflects the difficulty in upper bounding the contact process, even on a fixed graph.\smallskip

These models are asymptotically equivalent to inhomogenous random graphs, which are constructed by independent edges and so convenient to make dynamic by re-randomising edges individually. This has been investigated for power law degree distributions by  \cite{jacob2017contact} who found that now we do have a fast extinction case for a steep enough power law tail, and interestingly this case appears where the third moment converges. 
Self-duality,  which was useful on static graphs, 
continues to be useful on these non-adaptive dynamic graphs, as they have the same distribution run backwards in time.\smallskip

Recently \cite{ball2019stochastic} have considered adaptive dynamic graphs but with only edge removal and no replacement. Also, they consider the SIR infection {with removals which are permanent recoveries, and this} allows them to eliminate some dependence in the local picture of the spreading infection. Closer to our research, \cite{jiang2019sir} look at joint dynamics of homogenous graph and infection, which is similar to our context in that it features an Erd\H{o}s-R\'enyi graph with vertices moving if and only if they have an infected neighbourhood, but still the core infection model is SIR.
\smallskip

We go beyond these results for adaptive dynamic homogenous graphs by instead looking at the contact process (SIS). The SIS infection does not spread locally like a Galton-Watson tree as the SIR does. The addition of reinfection is why the SIS is in fact a stochastic upper bound of the SIR, and so of course it is harder to itself upper bound. We let the SIS infection spread as usual on the dynamic graph, and use the dynamics of vertex updating as in \cite{jacob2017contact,jacob2018metastability} where a vertex deletes all incident edges and generates an independent new neighbourhood. The modification which makes this an adaptive dynamic, then, is that rather than updating simply at independent Poisson times each vertex only updates (at fixed rate $\kappa>0$) while it sees infection in its neighbourhood.\pagebreak[3]

These graph dynamics for an infection model are very natural from a social dynamics perspective, and might at first seem similar to the non-adaptive dynamic version. 
Unfortunately, when the dynamics are adaptive, the time reversed graph model is very different and so there is no relationship between the cases of a single initially infected vertex and a fully infected graph. 
Our focus is on the more practicable case and therefore, recording the infected vertices at time~$t$ as $\xi_t \subset \{1,\dots,N\}$, we consider the process started from a single initially infected vertex, say $\xi_0=\{1\}$. We are interested in the size of the set $I^\ssup{N}_\infty$ of historically infected vertices or, more precisely, the occurrence of \emph{epidemic events} where this set grows to size $\Theta(N)$. Upper bounding the set $I^\ssup{N}_\infty$ will of course also bound the maximal size of the epidemic over~time.
\pagebreak[3]\smallskip

However, the modification has broken many useful mathematical features that were previously present. As well as losing self-duality for the contact process on this adaptive dynamic graph, we lose all types of monotonicity: an extra edge on the graph could cause an update and thus hinder the spread of infection; and so an extra infected vertex could in the same way lead to a reduction in the resultant size of $I^\ssup{N}_\infty$. The graph distribution is no longer stationary, firstly because the empty graph is an absorbing state for the dynamic. More significantly, while vertices are updated with the same (approximate) Poisson distribution with parameter~$\beta$, over time we will see the process spend more time in lower degree states because lower degree vertices are less likely to see an infected neighbour.\smallskip

As the adaptive context is technically challenging and little studied, in this article (like in previous SIR-based investigations) we focus on just the homogenous case of the adaptive dynamic, where connection probabilities are constant across the graph. 
In the homogenous context we will not encounter the issues of \cite{pastor2001epidemic} with polynomially large degrees, but we still cannot use the mean-field approximation because the network has become significantly dependent on the infection. 
Obtaining an upper bound for the contact process in the scale-free adaptive case, as has been achieved for static and independently dynamic cases, remains a challenging open problem.\smallskip

\pagebreak[3]

The main theorems determine when the epidemic event has asymptotic probability zero or when it has asymptotically positive probability. In each case we look to find an explicit region in our three parameters $(\lambda,\kappa,\beta)$ controlling the relative rate of infection, relative rate of updating and (initial) mean degree of the graph. Here we find that there is a phase transition in $\lambda$ on the adaptive homogenous graph, as there was on static and non-adaptive dynamic homogenous graphs, and we give explicit numerical estimates on the critical value as a function of the rate $\kappa$ of the graph dynamics. Further, we see that these dynamics can prevent the epidemic for arbitrarily large infection rate $\lambda$ and graph edge probability $\beta/N$ as long as $\kappa=\Omega(\lambda \beta)$ is large enough to stem the flow.\smallskip

One  way to show survival is by constructing an SIR infection inside the local limit of the joint infection and dynamic graph. {We improve on this simple argument when the individual local SIR processes are finite but large:} in this case, updates of infected vertices can infect other local environments and thus still give a supercritical process all together.\smallskip

For extinction when $\kappa$ is large, we use a percolation argument to show survival on a single local tree is impossible. The remaining hardest case is thus when $\kappa$ is small, where we introduce a reversible upper bound to the contact process which we use to prove a novel bound (Lemma \ref{MRCP_infection}) on the expected historical infection set of an SIS infection on a Poisson Galton-Watson tree.\smallskip

This bound requires $\lambda \beta< e^{-1}$ to have a finite expected number of historically infected vertices.
So, we will require $\lambda \beta< 0.21$ in the small $\kappa$ case to control that this expectation is finite and also small enough. 
From the local perspective of a vertex in the network, the original dynamics of changing neighbourhood whenever you see infection sounds like it should work to avoid epidemics, making epidemics asymptotically impossible for a larger parameter set than with non-adaptive updating. Indeed in simulations (Figure \ref{fig:scp_simulation}) we see this supported and that very likely the whole region $\lambda \beta< 1$ is included. Unfortunately, without any tighter upper bound we cannot hope to prove this claim.\smallskip

We examine also the non-adaptive version of this model, where as you might expect (and unlike in the adaptive case) we cannot prevent epidemics with large $\kappa$, whenever $\lambda \beta >1$. The phase transition in $\lambda$ for this model was seen implicitly in \cite{jacob2017contact} as the $\gamma=0$ case, but we look at the survival case to again give numerical bounds on the transition point in $\lambda$ which can be compared with the previous model to understand the effects of adaptivity. 

\section{Main results}

In this section, we first define the classical contact process on an evolving random graph
and then our adaptive version, where both the graph evolution also depends on the contact process. Finally, we state our main results and discuss 
the connection with the existing literature.

\begin{definition}[Non-adaptive contact process]
\begin{itemize}
    \item[(i)]
Suppose that $G = (G_t, t \geq 0)$ is a collection 
of graphs, such that for each $t \geq 0$, the graph $G_t$ has vertex set $[N] := \{1,\ldots,N\}$ and 
edge set $E_t$. 
The \emph{contact process} $(\xi_t)_{t \geq 0}$ with infection rate $\lambda > 0$ is the time-inhomogeneous Markov process, where at each time, each vertex can either be in state~$1$ (infected) or state~$0$ (healthy). 
At time~$t$, each infected vertex infects each 
healthy neighbour (as given by $E_t$) with rate $\lambda$. At rate $1$
each infected vertex recovers and becomes healthy again. 
The state $\xi_t$ of the process at time $t$
describes the infected vertices.
\item[(ii)] Given $N \in \N$ and  $\beta, \kappa > 0$, the \emph{dynamic Erd\H{o}s-R\'enyi graph} is the random graph 
$(G_t)_{t \geq 0}$ such that each graph $G_t$ has
vertex set $[N]$ and each pair of vertices $x\neq y$ is connected by an edge in $E_0$ independently with probability 
$\beta / N$.
Then, independently at rate $\kappa > 0$ for each vertex, the vertex $x$ updates by deleting all the 
edges incident to it and then independently reconnecting to each $y \neq x$ with probability $\beta / N$.
\item[(iii)] In the following, we will refer to the joint process 
$(G_t, \xi_t)_{t \geq 0}$ as the \emph{non-adaptive contact process} if $(G_t)_{t\geq 0}$ is 
a dynamic Erd\H{o}s-R\'enyi graph and $(\xi_t)_{t \geq 0}$ the contact process defined given $(G_t)_{t \geq 0}$.
\end{itemize}
\end{definition}

Note that the non-adaptive contact process is a (time-homogeneous) Markov process. Also, the dynamic Erd\H{o}s-R\'enyi graph is stationary and for a fixed $t> 0$,
the graph $G_t$ has the same distribution as the standard Erd\H{o}s-R\'enyi random graph.
We use the terminology non-adaptive to emphasize that the evolution of the graph is not influenced by the state of the contact process at any given time.

In contrast, we now consider a variation of this definition, where the random graph adapts to the infection by allowing each vertex  to update only when (at least) one of its neighbours is infected.

\begin{definition}[Adaptive contact process]\label{def_adaptive}
The \emph{adaptive contact process}
is the Markov process $(G_t,\xi_t)_{t \geq 0}$, where 
$G_t$ is a graph with vertex set $[N]=\{1,\cdots,N\}$ and edge set $E_t$ and $\xi_t$ is a subset of $[N]$ describing the infected vertices for each $t \geq 0$.
The initial graph $G_0$ is chosen as a standard 
Erd\H{o}s-R\'enyi random graph, i.e.\ 
$\{ x, y \} \in E_0$ independently with probability $\beta / N$ for each $x,y \in [N], x \neq y$ and $\xi_0$ is a given configuration of 
infected vertices.
Then, given $(G_t,\xi_t)$ for some $t \geq 0$:
\begin{itemize}
    \item[(i)] each vertex $x \in [N]$ updates at rate $\kappa$ if at least one of its neighbours is in $\xi_t$. In this case, all edges incident to $x$ are deleted and $x$ connects to each $y \neq x$ independently with probability $\beta / N$.
    \item[(ii)] 
    each infected vertex $x \in \xi_t$ infects each of its neighbours (given by $E_t$) with rate $\lambda$ and recovers (and becomes susceptible again) at rate $1$.
\end{itemize}
\end{definition}

With this adaptive definition, the dynamic graph becomes dependent on the infection and so  the graph $(G_t)_{t \geq 0}$ is \emph{not} Erd\H{o}s-R\'enyi in distribution for any $t>0$.

We distinguish between subcritical and supercritical by looking at the proportion of vertices that are infected
at some point during the infection. This is made precise in the following way.

\begin{definition}[Subcriticality and supercriticality]\label{def_superandsub}
Let $(G_t,\xi_t)_{t \geq 0}$ be either the adapted or the non-adapted contact process, where the graph $G_t$ has vertex set $[N]$.
For $t \geq 0$, 
let 
\[  I_t^\ssup{N} = \bigcup_{0 \leq s \leq t} \xi_s, \]
be the set of vertices infected by time $t$ and let
\[ I_\infty^\ssup{N} := \bigcup_{s \geq 0} \xi_s \]
be the set of vertices that have ever become infected.
\medskip

Assume that $|\xi_0| =1$.
We say that the family of contact processes is 
\emph{subcritical} if for every $\epsilon >0$ we have
\[ \lim_{N \rightarrow \infty}\mathbb{P} ( |I_\infty^\ssup{N}| > \eps N )=0.\]
Conversely, we say that the family is 
\emph{supercritical} if there exists $\epsilon > 0$ such that
\[
\liminf_{N \rightarrow \infty}\mathbb{P} ( |I_\infty^\ssup{N}| > \eps N  )>0.
\] 
\end{definition}

Our main result gives sufficient conditions 
for either sub- or supercriticality of the adaptive contact process.
\pagebreak[3]

\begin{theorem}\label{thm:1}
Consider the adaptive contact process started from $\left|\xi_0\right|=1$. 
\begin{enumerate}[label={(\alph*)},ref={\theproposition~(\alph*)}]
\item \label{subcriticality_theorem}
If $\beta, \lambda, \kappa \geq 0$ satisfy
\begin{equation}
\lambda \beta < 0.21
\quad
\text{or}
\quad
(2\beta-1) \lambda<\kappa,
\end{equation}
then the total number of infected vertices $\big( |I_\infty^{_{(N)}}| \big)_{N \in \mathbb{N}}$ are tight in $\mathbb{N}_0$ and in particular the contact process is subcritical.
\item \label{tree_supercriticality_thm}
If $\beta, \lambda, \kappa \geq 0$ satisfy
\begin{equation}\label{eq:suff_supercritical}
\frac{\lambda \beta}{1+2\kappa+\lambda}
 >
 \sqrt{1-\frac{\kappa (\lambda+\kappa)\big( 1-e^{_{-\frac{\beta \lambda}{\kappa+\lambda}}} \big)}{(1+\kappa)(1+2\kappa+\lambda)} },
\end{equation}
then the contact process is supercritical.
\end{enumerate}
\end{theorem}

Note that the expression on the right hand side of the final inequality~\eqref{eq:suff_supercritical} is  upper bounded by~$1$, which gives the simpler sufficient condition for supercriticality: $\lambda \beta > 1+2\kappa+\lambda$.
Also, note that when keeping $\beta$ fixed, then the critical $\lambda$ increases approximately linearly with $\kappa$, see Figure~\ref{fig:scp_simulation} for an illustration.

In~\cite{jacob2017contact} Jacob and M\"orters prove in particular that 
the non-adaptive contact process has a phase transition: for small $\lambda$ the process is subcritical and  for large $\lambda$ the process is supercritical. However, they state their results in terms of the time taken to reach the completely healthy state rather than the number of ever infected vertices. Also,  they do not give explicit sufficient bounds for either phase. 
Our second result 
extends these results to our definition of sub/supercriticality and also give
explicit bounds that allow us to compare the adaptive to the non-adaptive process.

\begin{theorem}\label{both_nonadaptive_theorems}
Consider the non-adaptive contact process started from $\left|\xi_0\right|=1$. 
\begin{enumerate}[label={(\alph*)},ref={\theproposition~(\alph*)}]\label{nonadaptive_theorem}
\item \label{nonadaptive_subcriticality_theorem}
If 
$\lambda \beta <1$ and $\kappa$ sufficiently large, then the total number of infected vertices
 $\big( |I_\infty^\ssup{N}| \big)_{N \in \mathbb{N}}$ is tight in $\mathbb{N}_0$, and in particular the contact process is subcritical;
\item \label{nonadaptive_submart}
If $\lambda \beta >1$ and $\kappa$ sufficiently large, then the contact process is supercritical.
\end{enumerate}
\end{theorem}

Comparing both our theorems, we see that the interaction between infection and graph structure has  a
major effect on where the phase transition happens. 
In the non-adaptive case, we see that for $\kappa$ large, the phase transition happens when 
$\lambda \beta  \approx 1$ and is in particular independent of $\kappa$. This  shows that the model becomes more like a mean-field model as $\kappa \rightarrow \infty$. 
This critical window also corresponds to the transition identified in~\cite{nam2019critical}, who look at the contact process on
a static random graph  but in the regime $\beta \ra \infty$.
See also Figure~\ref{fig:scp_simulation_nonadapt} for the results of simulations.
By these simulations, we see then that this large $\kappa$ condition is in fact only technical for the region of supercriticality, but still the critical line does curve towards a value less than $1$ for the $\kappa=0$ model.

For the adaptive model, Theorem~\ref{thm:1} shows that the 
location of the transition depends very much on $\kappa$. Moreover, the theorem 
suggests that the critical line should increase linearly with $\kappa$ (with lower bounds of $1/2$ and upper bounds of $2$ on the slope).
In Figure~\ref{fig:scp_simulation}, we plot the bounds obtained in our results as well as the results of simulations of the model. 
This suggests  the conjecture that there is a critical line asymptotically as $\kappa \ra \infty$.
By simulating the $1 \ll \beta \ll \kappa$ limit we observe that the large $\kappa$ slope of this adaptive critical line is in fact $\frac{\lambda \beta}{\kappa}\approx 0.7005 \pm 0.0001$.
This confirms that overall  the introduction of adaptivity is  always working to prevent the epidemic.
\begin{figure}[htp]
\centering
\includegraphics[width=0.7\textwidth]{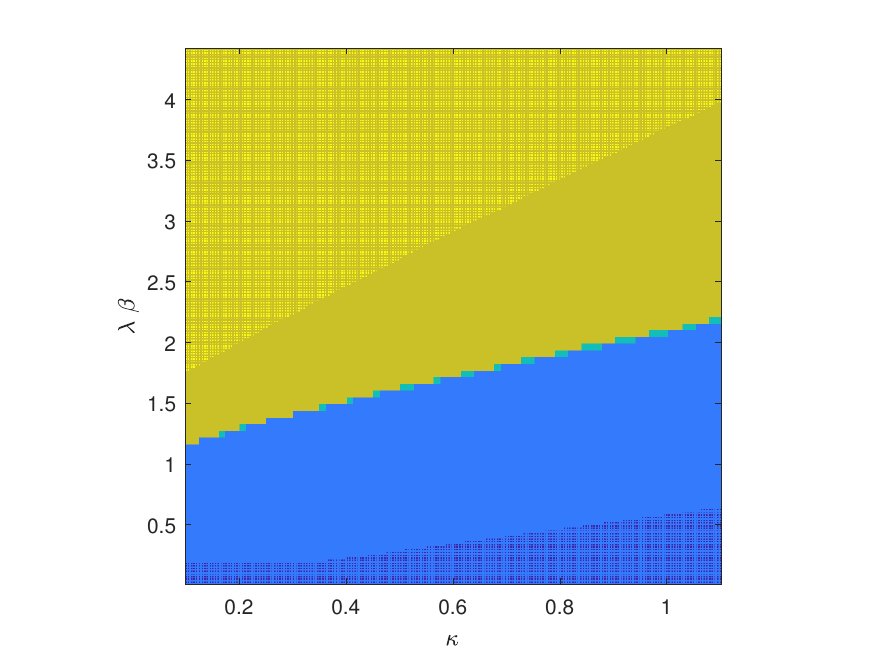}
\vspace{-1em}
\caption{For the adaptive dynamics with mean degree $\beta=3$, after $10^5$ samples in the yellow parameter regions we have two standard deviations of confidence for supercriticality, which we approximately identify as 98\% confidence. In the blue regions we have the same for subcriticality. The crosshatched regions within each are where our main theorems prove the sub- or supercriticality, in Theorems \ref{subcriticality_theorem} and \ref{tree_supercriticality_thm}.
}\label{fig:scp_simulation}
\end{figure}
\begin{figure}[htp]
\centering
\includegraphics[width=0.7\textwidth]{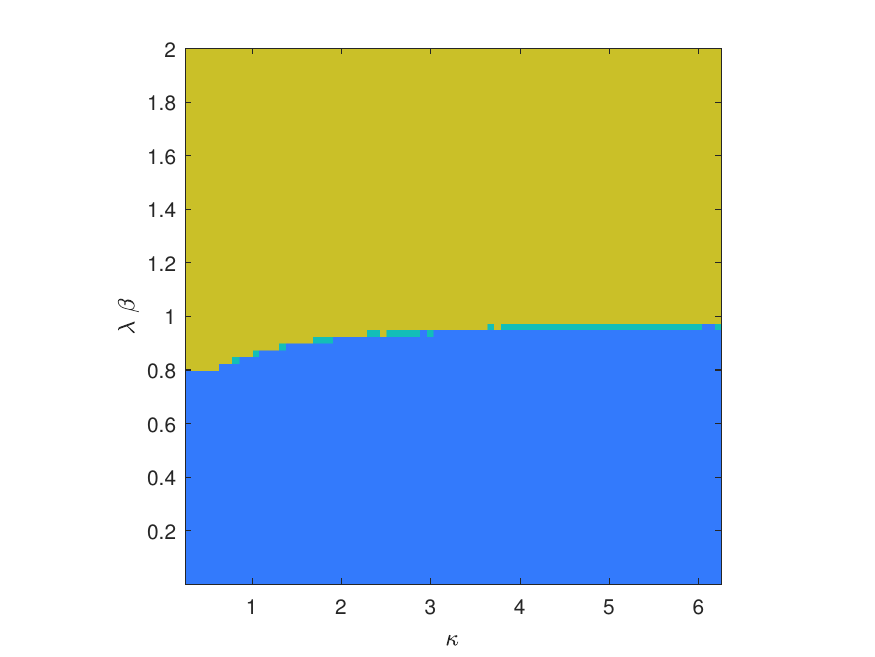}
\caption{For the non-adaptive dynamics with mean degree $\beta=3$, after $10^6$ samples in the yellow region we have three standard deviations (approximately over 99\%) of confidence for supercriticality and in the blue region for subcriticality. }\label{fig:scp_simulation_nonadapt}
\end{figure}

\begin{remark}[Time to extinction of the epidemic]
For the classic contact process on static Erd\H{o}s-R\'enyi random graphs, the 
distinction between a sub- and a supercritical phase is often formulated by looking at the 
time of extinction of the process. Moreover, in this classic case, there is a dichotomy between 
exponentially long extinction time (in the size $N$ of the underlying graph) versus a logarithmic 
time, see e.g.~\cite{bhamidi2019survival}. 
In our model, we handle difficult dependencies  with a local perspective to control only the initial phase of spreading and so we cannot phrase the results in terms of a long survival period of the epidemic. Moreover, unlike other models where this time is typically exponential in $N$ in the supercritical phase, we might instead expect a polynomial order of the last recovery time due to the constant removal of edges incident to infection.
\end{remark}
\pagebreak[3]

\section{Overview}

In Section~\ref{sec_nonadapt}
we study the non-adaptive dynamics, where the random graph and the contact process evolve independently.
To show the survival result of Theorem~\ref{nonadaptive_submart}  we lower bound the growth by a coupling with a multi-type branching process.

The most important step is to understand the local behaviour and to describe how quickly a vertex can pass on the infection on to other vertices.
We argue via the theory of quasi-stationarity for Markov chains that we can approximate the first infection of a dynamically changing number of children by a single exponential transmission time.\smallskip

To show extinction in the non-adaptive model
we use the mean-field upper bound which \cite{jacob2017contact} introduced for $\lambda \rightarrow 0$, with a different optimisation of the weights for our context. This provides a supermartingale which is stochastically above the number of infected vertices $|\xi_t|$, by which we infer a polynomial tail for the distribution of the historical infection set $|I_\infty|$.%
\smallskip

In Section \ref{sec_adapt} we move on to the adaptive model, which is our primary interest.  Our main tool to study the adaptive model is its approximation by an idealised model, the \emph{contact process on an evolving forest} (CPEF), which is introduced in Section~\ref{chap_contact}. {In this model, to reflect that typical distances in the network diverge, updating events of a vertex in a tree correspond to the creation of a new tree with that vertex at the root. By seeing the new tree as offspring, we identify this as a tree of trees, or \emph{meta-tree} in which to every  meta-tree ``vertex'' corresponds to an independent infection process on a Poisson Galton-Watson \emph{host tree}. As the meta-tree itself is a Galton-Watson tree we know that the CPEF survives if and only if its meta-tree offspring mean exceeds one.  
The CPEF is not only an important tool in the proofs, but also facilitates simulation of the local limit of the dynamic graph (with natural modifications in the non-adaptive case) as by just simulating one tree in the meta-tree we can determine its sub- or super-criticality, see also Figures~\ref{fig:scp_simulation} and~\ref{fig:scp_simulation_nonadapt}.}
\smallskip

For a lower bound on the CPEF meta-tree offspring mean we use an SIR lower bound on each host tree of the CPEF. As the set of vertices ever infected by an SIR process on a Galton-Watson tree is identical to an associated Galton-Watson subtree, we can check when the expected number of infected children exceeds~$1$ in this subtree and infer an infinite offspring number of the meta-tree with positive probability. This
gives survival under the condition $\lambda \beta > 1 + 2\kappa + \lambda$.
 The condition of Theorem~\ref{tree_supercriticality_thm}  
 is weaker and uses that these local SIR lower bounds can in fact be subcritical and still guarantee supercriticality of the meta-tree, as long as the SIR offspring mean is sufficiently close to one.\smallskip

In the second half of Section \ref{sec_adapt} we show extinction of the CPEF by considering two distinct cases. One of the few tools we have to control the spread of the infection is the \emph{Subtree Contact Process} (SCP) that we introduce in Definition \ref{def_SCP}, which bounds the contact process on a rooted tree by disallowing all recoveries that would disconnect the infection set and thus constructing a reversible Markov process. {We find that this upper bound  hits the empty infection in finite expected time only when $\lambda \beta <e^{-1}$. This is a limit therefore on the size of the region of subcriticality. 
\smallskip

Fortunately, analysis of this process on a generic member of the meta-tree is very feasible when we look at the annealed context, averaging over the Poisson-Galton-Watson host tree. 
 To achieve this for smaller $\kappa$, we introduce the \emph{Slowed SCP} (SSCP), which has the same paths as the SCP but leaves any infection state at a rate slowed exponentially in the size of the infection set. 
 Bounding the expected recovery time for the SSCP leads to the bound of Lemma \ref{MRCP_infection} which has some independent interest: we show that for the usual contact process $(\xi_t)_t$ on a Poisson($\beta$)-Galton-Watson tree with initially one infected vertex at the root $\mathbbm{o}$, the historical infection set has a bounded mean for small enough $\lambda \beta$ and in particular \[\lambda \beta< \mathbb{E}
 \bigg(\Big|\bigcup_{t=0}^\infty \xi_t \setminus\{\mathbbm{o}\}\Big|\bigg)
 < 2\lambda \beta \quad \text{when } \lambda \beta <\tfrac{1}{15}.\]
 %\JFtodo{The $1/15$ comes from solving Lemma \ref{MRCP_infection}$=1+2\beta\lambda$ to get \[\beta \lambda =\frac{15}{6 e^3 \sqrt{\frac{6}{\pi }}+20 e}>\frac{1}{15}\]
 %so this is a weakened result for readability, maybe we could say something like "by solving a quadratic inequality in $\beta \lambda$"?}
This lemma then produces the first condition of Theorem \ref{subcriticality_theorem},  with no dependence on $\kappa$. %\smallskip
The second condition of Theorem \ref{subcriticality_theorem} supercedes the first for larger $\kappa$, and follows from} an easier argument by percolation, which uses that if your local infection consists of only one vertex, then it cannot update and so the expected size of each local infection set must be at least $2$ for survival. \smallskip

Both survival and extinction results proved for the CPEF are used to obtain the main Theorems \ref{tree_supercriticality_thm} and \ref{subcriticality_theorem}, respectively, by coupling to the process on the network with $N\rightarrow \infty$ vertices.
The coupling is achieved by means of an exploration algorithm described in  Section~\ref{sec_explore_infection}.
In this exploration, while $|I_t^{_{(N)}}|\leq \epsilon N$, we are guaranteed at least a binomial number of unrevealed neighbours with parameters $\lceil(1-\epsilon) N\rceil$ and 
${\beta}/{N}$ in each vertex exploration or update. This can 
be coupled to the Poisson distribution of the offspring 
in the idealised CPEF (with slightly reduced mean) to give a stochastic 
lower bound on the true infection, which however survives with positive probability. 
\smallskip

In principle, the coupling should be harder for the reverse direction where we require an upper bound on the historical infection set. However, because in this case the quantity $|I_\infty|$ is tight for the CPEF, we can show that the CPEF and network versions can be coupled exactly with high probability in $N$.  

\section{The contact process}\label{chap_contact}

It is a classical result that the standard contact process on a static graph $G=([N],E)$ can be constructed
using a \emph{graphical construction}.
We construct $2|E|$ independent Poisson processes of rate $\lambda$ which describe the potential infection times for each directed edge.
Then, $i$ infects $j$ at time $t$ at the times when the Poisson process attached to $(i,j)$ rings at $t$ and also $i$ is infected at $t$.
Further we generate $N$ independent Poisson processes of rate $1$, which give the recovery times of each vertex.

With this conception of the model, we can construct couplings between the contact process and versions with extra infection added or recovery taken away.
Classicaly this construction is used, see e.g.~\cite[Page 32]{liggett1999stochastic}, to show monotonicity of the contact process with respect to graph containment, initial condition containment or increased infection rate. 
A similar construction, also works for the non-adaptive model, see \cite[Proposition 3]{jacob2018metastability}.

Even in the adaptive case, we can use a graphical construction to generate the infections by 
constructing Poisson processes for each pair of vertices $i,j $ and then only taken into account 
infections if they happen at a time when the vertices $i$ and $j$ are actually neighbours.
However, we can no longer use this construction to show monotonicity properties (because they fail e.g.\ in the infection rate).
Nevertheless, the graphical construction will be used to construct couplings between different versions of our model.

\subsection{Exploring with the infection}\label{sec_explore_infection}

We now simultaneously look at the contact process on the adaptive dynamic network model defined in Definition \ref{def_adaptive}
 and the non-adaptive model. We define the dynamics of the coupled infection and graph, from a local perspective, to better understand how spreading happens. For this purpose we define an algorithm which, by time $t$, generates a part of the graph (consisting of vertices, edges and half edges) on which the infection acts. Vertices can have two types, \emph{revealed}  or \emph{unrevealed}, they can switch type and the set of revealed vertices at time $t$ is denoted~$R_t$. 
\pagebreak[3]

\begin{tcolorbox}[enhanced,breakable]
\begin{itemize}[leftmargin=*]
\item
At time $0$ we have
\[
\xi_0=I_0=R_0=\{1\}
\]
and we reveal the vertex $1$ by independently connecting it by an oriented half-edge to every other vertex with probability $\beta/N$. The infection will be passed on from vertex~1 to each of these vertices with rate $\lambda$ as long as vertex~1 remains infected and does not update.
This is enough local information to run the contact process and the graph process, as in the adaptive model a vertex cannot update without being adjacent to an infected vertex and in the non-adaptive model updates outside the neighbourhood of revealed vertices do not play a role, by stationarity.
\item
Continuing recursively, when an infection is passed through a half-edge to a yet unrevealed vertex $v$ at time~$t$ we add it to the set $R_t$ of revealed vertices and the half-edge becomes an edge.
Its indegree $\de_t^-(v)\geq 1$ is the total number of revealed vertices with an edge leading to it and by independently creating an oriented half-edge to any other unrevealed vertex with probability $\beta/N$ we sample an outdegree 
\[
\de_t^+(v) \sim \operatorname{Bin}\bigg(
N-|R_t|,\frac{\beta}{N}
\bigg).
\]
\item
At an \emph{update of an infected vertex} $v$ at time $t$ (which is necessarily revealed) we delete every edge pointing into it, every edge pointing out (to a revealed vertex) and every half-edge pointing out (to an unrevealed vertex). Then we generate a new outdegree by independently forming a half-edge to every unrevealed vertex with probability $\beta/N$, and a new indegree by independently forming an edge from every other revealed vertex, again with probability $\beta/N$.
\item
At an \emph{update of an uninfected vertex}, if it is revealed, it returns to being unrevealed and is removed from the set $R_t$, and if it is unrevealed we remove any half-edges pointing to it; in either case we independently with probability $\beta/N$ add a half-edge from each revealed vertex to this vertex but we do not generate an outdegree. 
\end{itemize}
\end{tcolorbox}

In the following, we will sometimes call the vertex 
that becomes infected or updates while infected the \emph{parent} of any of 
the vertices that the newly created directed half-edges point to (which will be the \emph{children}).

In this construction, we have in a sense contained the nonstationarity: For any time $t>0$ the edges adjacent to a vertex in $R_t$ have a probability of presence dependent on the contact process, but other edges are independently present with probability $\beta/N$.
In particular, we can use $R_t \subseteq I_t^\ssup{N}$ to say that before the epidemic event $\{I_\infty^{_{(N)}}\geq \epsilon N\}$ each vertex when it was newly revealed had indegree $\de^-\geq 1$ including an edge starting at the vertex that infected it, and 
\[
\de^+ \succeq \operatorname{Bin}\bigg(
\left\lceil (1-\epsilon)N \right\rceil, \frac{\beta}{N}
\bigg)
\]
many edges leading to unrevealed vertices that have never been infected.
\smallskip

\pagebreak[3]

We now describe an idealized model which will describe the local neighbourhoods in the $N$ large limit
at the onset of the epidemic.
First note that  local neighbourhoods in Erd\H{o}s-R\'enyi random graphs are treelike
and converge locally to a Galton-Watson tree with a $\operatorname{Pois}(\beta)$ offspring distribution.
Moreover, the typical distance for the  giant component is $\Theta_{\mathbb{P}}(\log N)$, 
so that updates are expected to connect vertices to new vertices far away in the graph. 
This intuition leads to a very natural approximating dynamic forest model, which 
we will show describes the initial phase $|R_t| \ll N$ well.

\begin{definition}[Contact Process on an Evolving Forest]\label{CPEF_defn}
Generate a countably infinite set of i.i.d. Galton Watson trees with offspring distribution $\operatorname{Pois}(\beta)$, labelled $T^u$ for each label $u$ in the Ulam-Harris tree. At time $t=0$ the \emph{Contact Process on an Evolving Forest} (CPEF) consists of the tree~$T^\emptyset$, with root
\[
v_\emptyset^\emptyset \in T^\emptyset
\]
and initially this is the only infected vertex.  We number this vertex as $1$ and continue numbering the other vertices of $T^\emptyset$ in the lexicographical (breadth first) ordering. On this graph we start the contact process dynamics: infected vertices infect each of their neighbours independently at rate $\lambda$, while recovering at rate $1$. Vertices in $T^\emptyset$ update at rate $\kappa$ if and only if they have an infected neighbour. 
When the vertex at lexicographical order $k$ updates in $T^\emptyset$, it is deleted in that graph and the tree $T^k$ added to the forest. The vertex is
identified with the root of $T^k$, so that it will be infected at that point in time iff it is currently infected. 
\smallskip

{Given the CPEF and the infection at time $t$,} each vertex in the CPEF updates at rate $\kappa$ if and only if it has an infected neighbour. 
Thus when the vertex at lexicographical order $k$ updates in $T^u$, it is deleted from the CPEF and $T^{u,k}$ added to the CPEF and the vertex identified
with its root. Then, vertices in $T^{u,k}$ are also numbered by the lexicographical ordering so that if some vertex in position $m$ in that tree does update, it is deleted and identified with the root of $T^{u,k,m}$, which itself is added to the forest, and so forth. We sketch this structure in Figure~\ref{fig:scp_diagram}.
\end{definition}

\begin{figure}[ht]
\centering
\centerline{\includegraphics[width=0.8\textwidth]{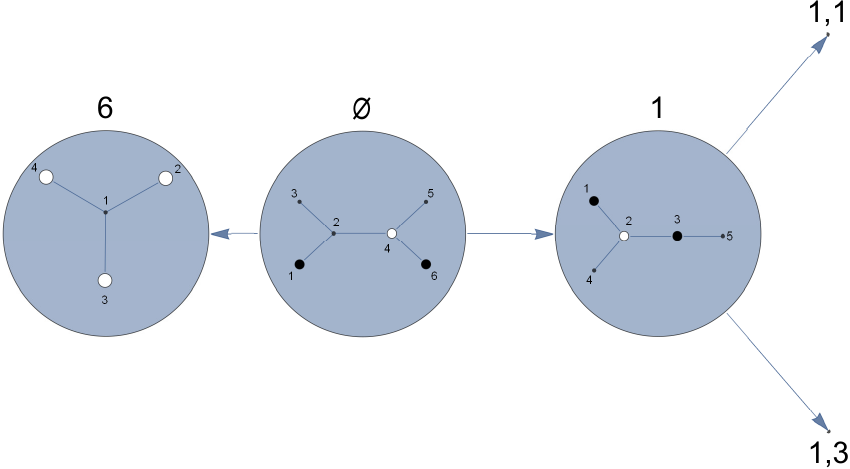}}
\vspace{2em}
\caption{An illustration of the first three historical infection sets, which happen to be finite, in a large meta-tree. The circled vertices are those which updated in the local infection process, and of those only the black circles were infected at the update time (note that there can only be one update time). Then, each black vertex is ``moved'' to the root of a new i.i.d. environment and thus gives rise to a new vertex in the meta-tree.}\label{fig:scp_diagram}
\end{figure}

Note that trees added to the CPEF may eventually become disconnected by the graph evolution, but always remain a subgraph of the initial tree and 
so will never contain a cycle. We write $I_\infty^u$ for the vertices in $T^u$ which ever become infected. Also, crucially, each tree in this CPEF model is independently distributed like the first tree. Thus the whole process forms a meta-Galton-Watson tree, in which each individual simulates an infection in its own Galton-Watson environment -- and has offspring given by the number of updates observed at infected sites. We can think of the CPEF as a version of the dynamics that is obtained from the dynamics on the finite network by taking  $N\rightarrow\infty$ in the local weak limit.

\section{The non-adaptive dynamics}\label{sec_nonadapt}

In this section we prove Theorem~\ref{nonadaptive_theorem} which shows a phase transition for the non-adaptive contact process depending on whether $\lambda \beta$ is greater or smaller than $1$,
as long as the update rate $\kappa$ is sufficiently large.
In this case, the random graph dynamics are  independent of the infection and hence we are able to apply  existing techniques to show the existence of a subcritical phase
in Section~\ref{sec:subcritical_non-adaptive}. In particular, we will adapt the martingale techniques from~\cite{jacob2017contact} to prove Theorem~\ref{nonadaptive_subcriticality_theorem}.
For the proof of Theorem~\ref{nonadaptive_submart}, 
we first analyse the effect of the graph dynamics in the local neighbourhood of a given vertex, which we then, in a second step, use 
to couple our process with a branching process that survives with positive probability.

\subsection{Quasi-stationarity}\label{metastability_section}

In this section, we recall some basic definition and facts about Markov chains that we will then 
use in the next section to study the effect of the graph dynamics on the local neighbourhood of 
a given vertex.

Consider an irreducible and reversible continuous-time Markov chain $X = (X_t)_{t \geq 0}$ on a finite state space $V$ 
with generator matrix $Q$ and transition probabilities
\smash{$p^{_{(t)}}_{x,y}$}, for $x,y\in V$ and  $t\ge 0$. 
If $X$ has an invariant distribution  $\pi$, then we can measure mixing by the total variation distance between the walker from any two states
\begin{equation}\label{eq:dist_stat}
\bar{d}(t):=\frac{1}{2}\max_{x,y \in [N]}\| p^{(t)}_{x,\cdot}-p^{(t)}_{y,\cdot} \|_1.
\end{equation}

The \emph{mixing time}
 $t_{\rm mix}$ is then defined as 
\begin{equation}\label{eq:mix_time}
t_{\text{\textnormal{mix}}}:=\min \Big\{ t\geq 0 : \overline d(t)\leq \tfrac{1}{e}\Big\},
\end{equation}
and the relaxation time as
\[ 
t_{\text{\textnormal{rel}}}:=\max \left\{ \frac{1}{\lambda} : \lambda \text{ a positive eigenvalue of } -Q \right\}. \]

We now consider a  not necessarily reversible and irreducible Markov chain and assume that $V = S \cup \{ \dagger\}$, where $\dagger$ is an absorbing state. 
We say that a distribution~$\alpha = (\alpha(i))_{i \in V}$ is a \emph{quasi-stationary distribution} for a Markov chain $X$ on $V$  with generator matrix $Q=(Q(i, j))_{i,j \in V}$ 
and absorbing set $\{\dagger\}$ 
if there exists $\rho > 0$ such that $\alpha$ satisfies
\begin{equation}\label{eq_first_metastab_def}
\begin{split}
\alpha(\dagger)& =0, 
\\
\sum_{i \in S} \alpha(i) Q(i,j)& =-\rho \, \alpha(j) \quad \text{ if } j \in S, \\
\sum_{j \in S} \alpha(j) & =1. \quad
\end{split}
\end{equation}
If we write $Q^S$ for the matrix $Q$ restricted to $S$ then \cite[Section 3.6.5]{aldous-fill-2014} tells us that these equations have a unique solution if the off-diagonal entries of $Q^S$ define an irreducible Markov chain on $S$. In this case,  $\rho$ is the spectral radius of $Q^S$. 
Moreover, starting the Markov chain $X$ from $\alpha$, the hitting time of the state $\dagger$ is precisely exponentially distributed with rate parameter~$\rho$.
From a general initial distribution on $S$, writing $T_\dagger$ for the hitting time of $\dagger$, applying \cite[Remark  after Theorem 3.33]{aldous-fill-2014} tells us that the law of the process conditioned on $\{T_\dagger>t\}$ will converge weakly as $t\rightarrow\infty$ to the quasi-stationary distribution $\alpha$.
\smallskip

The following lemma will be used in the later  applications of this theory: the intuitive idea is that quasi-stationary conditioning biases the stationary distribution towards states with lower escape rates.

\begin{lemma}\label{prop_metastable_flow}
Consider a continuous time Markov chain on state space $V=S \cup \{\dag\}$ with generator matrix $Q$ such that $Q^S$ is irreducible on $S$. For each $s \in S$, write $r(s):=Q(s,\dag)$ for the escape rate from $s$. If $S$ can be partitioned into $S=A \cup B$ with
\[
\sup_{a \in A} r(a) \leq \inf_{b \in B} r(b),
\]
then the unique quasi-stationary distribution $\alpha$ satisfies
\[
\sum_{a \in A} \alpha(a) Q(a,B)
\geq
\sum_{b \in B} \alpha(b) Q(b,A), 
\]
where we write $Q(a,C) = \sum_{c \in C} Q(a,c)$ for $a \in V, C \subset V$.
\end{lemma}

\begin{proof}
Let $\alpha$ denote the unique  quasi-stationary distribution.
We first consider the Markov chain $\tilde X$ on $S$ with 
transition rates given by
\[
\tilde{Q}(i,j)=Q(i,j)+r(i)\alpha(j)
\]
for any $i,j \in S, i \neq j$.
This corresponds to modifying the original Markov chain so that at every attempted
move to $\dagger$, the Markov chain is  instantaneously moved to 
a location in $S$ chosen according to $\alpha$. Using the definition~\eqref{eq_first_metastab_def}, 
a straight-forward calculation (which also reveals that $\rho = \sum_{j \in S} \alpha(j) Q(j,\dagger)$) shows that $\alpha$ is the invariant 
distribution for $\tilde X$.
\smallskip

Next, we  define the corresponding \emph{collapsed} Markov chain, cf.\ \cite[Section 2.7.3]{aldous-fill-2014}, 
as the Markov chain with states $\{ A,B \}$ and transition rates
\[
\begin{split}
\tilde{q}_{A B} &= \sum_{a \in A} \frac{\alpha(a)}{\alpha(A)}Q(a,B) + \sum_{a \in A} \frac{\alpha(a)}{\alpha(A)}r(a)\alpha(B),\\
\tilde{q}_{B A} &= \sum_{b \in B} \frac{\alpha(b)}{\alpha(B)}Q(b,A) + \sum_{b \in B} \frac{\alpha(b)}{\alpha(B)}r(b)\alpha(A).
\end{split}
\]
As $\alpha$ is invariant for $\tilde X$, it follows 
that $(\alpha(A), \alpha(B))$ is invariant for the collapsed chain. 
As a two-state Markov chain, it satisfies  detailed balance  so that
\[
\sum_{a \in A} \alpha(a) \left(
Q(a,B) + r(a) \alpha(B)
\right)
=
\sum_{b \in B} \alpha(b) \left(
Q(b,A) + r(b) \alpha(A)
\right) .
\]
We conclude that
\[
\sum_{a \in A} \alpha(a)
Q(a,B) +
\alpha(A) \alpha(B) \sup_{a \in A} r(a) 
\geq
\sum_{b \in B} \alpha(b) 
Q(b,A) +
\alpha(A) \alpha(B) \inf_{b \in B} r(b).
\]
Recalling the assumption that $\sup_{a \in A} r(a) \leq \inf_{b \in B} r(b)$, we deduce that
\[
\sum_{a \in A} \alpha(a)
Q(a,B)
\geq
\sum_{b \in B} \alpha(b) 
Q(b,A), 
\]
as required.
\end{proof}

\subsection{Supercriticality for the non-adaptive dynamics}

The main challenge in describing the local behaviour of an infected
vertex trying to spread the infection
is that its neighbourhood is constantly updating (especially as we assume that the update rate $\kappa$ is large). 
In the following proposition, we show that we can couple the first infection time of a neighbour  to an exponential random variable with an effective rate of spreading. \smallskip

\pagebreak[3]

We will use this proposition at the point, where we have revealed
at most $\eps N$ vertices and we want to describe the local
dynamics of a newly revealed infected vertex $v$.
If we are only looking at passing the infection to its unrevealed 
neighbours, then it initially attempts to infect ${\rm Bin}_{(1-\eps)N,\beta/ N} \approx {\rm Pois}({\beta_*})$  such vertices, where ${\beta_*} = (1-\eps) \beta$. 
Each of these neighbours updates at rate $\kappa$ reducing the degree by $1$.
Moreover, every unrevealed vertex in the network updates at rate $\kappa$
and then connects to $v$ with probability $\beta / N$, which asymptotically leads to the increase of the degree of $v$ at rate $(1-\eps) \beta \kappa$. Finally, when $v$ updates 
then the number of its neighbours is  chosen afresh from 
a ${\rm Pois}((1-\eps) \beta)$ distribution.

To be able to carry out the coupling of these idealized dynamics 
to the full process in Proposition~\ref{prop:coupling_branching} below, we  also have to restrict the neighbourhood size to some large constant $L \in \N$. 
In particular, we will use the truncated Poisson distribution 
$P_L^{\beta_*}$, which is the probability distribution on $\{ 0, \ldots,  L \}$ with probability weights satisfying.
\begin{equation}\label{eqn:truncated_Poisson}
    P_L^{\beta_*}(x)
\propto
 \frac{{\beta_*}^x}{x!}\mathbbm{1}_{x\leq L} . 
\end{equation}
We will also use that $P_L^\beta$ is stochastically dominated by $ {\rm Pois}(\beta)$.

\begin{proposition}\label{lemma_dynamic_star}
For $L \in \N$, consider a Markov chain $(\de^+(t))_{t>0}$ with state space $\{0, \ldots, L\} \cup \{ \dagger\} $. 
Suppose $\de^+(0)$ has law  $P_L^{\beta_*}$ 
and the 
transition rates are given as
\begin{equation}\label{eq_degree_dynamic}
\begin{array}{ll}
\de^+ \rightarrow \de^+ - 1 & \text{ at rate } \kappa \de^+ \text{ if } \de^+\neq 0, \\
\de^+ \rightarrow \de^+ + 1 & \text{ at rate } \kappa {\beta_*} \text{ if } \de^+\neq L, \\
\de^+ \rightarrow \de & \text{ sampled independently from } P_L^{\beta_*} \text{ at rate } \kappa, \\
 \de^+ \rightarrow \dagger & \mbox{ at rate } \lambda \de^+ ,
\end{array}
\end{equation}
and where $\dagger$ is an absorbing state.
Let $T$ be the first hitting time of $\dagger$.
Then, $T$ can be coupled to an exponential random variable $E$
with parameter $\lambda {\beta_*}$ such that
for any $\epsilon\in (0,1)$, there are $\kappa_0$ and $L_0$ such that, for $\kappa>\kappa_0, L>L_0 $ we have $$\mathbb{P}\left(
T \geq E
\right)
\leq \epsilon.
$$
\end{proposition}

The hitting time $T$ in the proposition has the interpretation of the 
first time that the centre of a star with a dynamically evolving 
neighbourhood infects one of its children.

\begin{proof}
Let $\pi = P_L^{\beta_*}$ and note that this distribution satisfies the detailed balance equations for the dynamics restricted to the states $\{ 0, \ldots, L\}$.
Let $\alpha_\kappa$ be the quasi-stationary distribution for the absorbing state $\dagger$ as defined 
in Section~\ref{metastability_section}.
The strategy of the proof is to first show that 
$\alpha_\kappa$ is close to $\pi$ for $\kappa \ra \infty$ and also that $\rho$ is close to $\lambda {\beta_*}$. 
In a final step, we will then use the fact that $T$ has an exponential distribution with rate $\rho$ when the Markov chain is started from $\alpha_\kappa$ to show our coupling. In order to show that $\alpha_\kappa$ is close to ${\beta_*}$, 
we will use moment estimates and a spectral decomposition. 
\smallskip

\emph{Step 1. Moment estimates.}
Let $A_\kappa$ and $\Pi$ be random variables with distributions $\alpha_\kappa$ and $\pi$ respectively. We claim that 
\begin{equation}\label{eq:moment_qs} \E(A_\kappa) \leq {\beta_*}, \quad \mbox{and} 
\quad \E(A_\kappa^2) \leq \frac{4}{3} {\beta_*} + {\beta_*}^2 . \end{equation}

Recall from Section~\ref{metastability_section} that the defining equation for $\alpha_\kappa$ is that for each $j \in \{0\}\cup [L]$ 
\begin{equation}\label{eq_metastable}
\kappa \sum_{i=0}^L \alpha_\kappa(i) \bar q_{ij}-\lambda j \alpha_\kappa(j)
=
-\rho \alpha_\kappa(j).
\end{equation}
where $\bar Q=(\bar q_{ij})_{i,j \in \{0,\ldots,L\}}$ is the generator of the dynamics restricted to $\{0,\ldots,L\}$ in the case $\kappa=1$.
Let $x \in [L]$ and define
\[
A_x=[0,x-1] \cap \mathbb{Z},
\qquad
B_x=[x,L] \cap \mathbb{Z}.
\]

The  rate to jump to $\dagger$ at any state in $B_x$ is at least $\lambda x$, whereas in $A_x$ it is at most $\lambda(x-1)$. So we can apply Lemma \ref{prop_metastable_flow} to these two sets. Using that 
the generator matrix of the full chain has entries $\kappa \bar q_{ij}$ for $i, j \in \{0,\ldots,L\}, i \neq j$, we thus obtain
\[
\kappa \sum_{i=1}^{x-1} \sum_{j=x}^L \alpha_\kappa(i) \bar q_{ij}
\geq
\kappa \sum_{i=1}^{x-1} \sum_{j=x}^L \alpha_\kappa(j) \bar q_{ji}.
\]
Inserting the rates given in \eqref{eq_degree_dynamic} yields
\[
\alpha_\kappa(A_x)\pi(B_x) + {\beta_*} \alpha_\kappa(x-1) \geq \alpha_\kappa(B_x)\pi(A_x)+x\alpha_\kappa(x) .
\]
Using that both $\alpha_\kappa$ and $\pi$ are probability measures 
on $\{0,\ldots, L\}$, we arrive at
\begin{equation}\label{eq_alpha_ineq}
\pi(B_x) + {\beta_*} \alpha_\kappa(x-1) \geq \alpha_\kappa(B_x)+x\alpha_\kappa(x).
\end{equation}

Summing \eqref{eq_alpha_ineq} gives
\[
\begin{split}
2\mathbb{E}(A_\kappa)&=\sum_{x=1}^L \left(\alpha_\kappa(B_x)+x\alpha_\kappa(x) \right)\\
&\leq \sum_{x=1}^L \left(\pi(B_x) + {\beta_*} \alpha_\kappa(x-1) \right)
=\mathbb{E}\left(\Pi\right)+{\beta_*}-{\beta_*}\alpha_\kappa(L)
\leq \mathbb{E}\left(\Pi\right)+{\beta_*}.
\end{split}
\]
Recalling that $\pi = P_L^{\beta_*}$ is stochastically dominated by a $\operatorname{Pois}({\beta_*})$ distribution, we have $\mathbb{E}(A_\kappa) \leq {\beta_*}$, 
which is the first claim in~\eqref{eq:moment_qs}.
\pagebreak[3]
\smallskip

To bound the second moment, we will use the following general fact
that holds for any positive integer valued random variable $X$,
\[
\begin{split}
\sum_{x=1}^\infty x \mathbb{P}(X \geq x)
& =
\sum_{x=1}^\infty x \sum_{y=x}^\infty \mathbb{P}(X =y)
=
\sum_{y=1}^\infty \mathbb{P}(X =y) \sum_{x=1}^y x 
=
\sum_{y=1}^\infty \mathbb{P}(X =y) \left(\frac{y(y+1)}{2}\right)\\
&=\frac{\mathbb{E}(X)+\mathbb{E}(X^2)}{2}.
\end{split}
\]
Indeed, using again that $\pi \preceq {\rm Pois}({\beta_*})$ we have   
\[
\begin{aligned}
\sum_{x=1}^L x \left(\pi(B_x) + {\beta_*} \alpha_\kappa(x-1) \right)
& =
\frac{\mathbb{E}(\Pi)+\mathbb{E}(\Pi^2)}{2}+{\beta_*} \big(1+\mathbb{E}(A_\kappa)-(L+1)\alpha_\kappa(L)\big)\\
& \leq
2{\beta_*} +\frac{{\beta_*}^2}{2} + {\beta_*} \mathbb{E}(A_\kappa).
\end{aligned}\]
Moreover, by~\eqref{eq_alpha_ineq},
\[ \begin{aligned}
\sum_{x=1}^L x \left(\pi(B_x) + {\beta_*} \alpha_\kappa(x-1) \right)
&\geq
\sum_{x=1}^L x \left(\alpha_\kappa(B_x)+x\alpha_\kappa(x) \right)\\
& =
\frac{\mathbb{E}(A_\kappa)+\mathbb{E}(A_\kappa^2)}{2}+\mathbb{E}(A_\kappa^2)
=
\frac{1}{2}\mathbb{E}(A_\kappa)+\frac{3}{2}\mathbb{E}(A_\kappa^2)
\end{aligned}
\]
Combining the last two displays,  we can conclude that \[ \mathbb{E}(A_\kappa^2) \leq \frac{4}{3}{\beta_*} + \frac{{\beta_*}^2}{3} + \frac{2}{3}{\beta_*} \mathbb{E}(A_\kappa)\leq \frac{4}{3}{\beta_*} + {\beta_*}^2,\]
which completes the proof of~\eqref{eq:moment_qs}.
\smallskip

\emph{Step 2: Spectral bounds:} We claim that $\alpha_\kappa(i) \ra \pi(i)$
for every $i \in \{ 0, \ldots, L\}$,
as $\kappa \ra \infty$. 

To show the convergence of $\alpha_\kappa$, we use the spectral representation tools for reversible Markov chains, see e.g.~\cite[Chapter 3.4]{aldous-fill-2014}. 
Define a symmetric matrix $M $ with entries 
$$M_{ij} = \pi_i^{1/2} \bar q_{ij} \pi_{j}^{-1/2}
\text{ for $i, j \in S := \{ 0,\ldots, L \}$.}$$
Then, the spectral theorem gives us an orthonormal basis of (left) eigenvectors $u_j, j = 0, \ldots, k$ corresponding to the eigenvalues $\la_0 = 0 > \la_1\geq \ldots \geq \la_K$ of $M$. The matrices $M$  and $Q$ have the same 
eigenvalues and moreover $v_k(i) := \pi_i^{1/2} u_k(i)$ defines a set of
orthonormal eigenvectors of $\bar Q$ when using the inner product
\[ \langle x, y \rangle_{1/\pi} = \sum_{i = 0}^L x_i y_i  / \pi(i) , \quad x, y \in \R^{K+1},  \]
and where will denote the associated norm by $\| \cdot \|_{1/\pi}$.
We will also use that $v_0 = \pi$.

Next we find bound the eigenvalue $\la_1$ by bounding the mixing time of the Markov chain on~$S$ with generator $\bar Q$. As the dynamics resamples afresh from $\pi$ after an ${\rm Exp}(1)$-distributed time~$Y$, we can couple the Markov chain started from two different starting points so that the trajectories agree completely after this time.  In particular, using
the definition~\eqref{eq:dist_stat} we have that
\[ \bar d(1) = \frac 12 \max_{x,y} |p^{(1)}_{x,\cdot} - p^{(1)}_{y,\cdot}|  \leq \p ( Y \geq 1) = e^{-1},  \]
so that by definition~\eqref{eq:mix_time}, we have $t_{\rm mix} \leq 1$. 
By \cite[Lemma 4.23]{aldous-fill-2014}, it follows that $t_{\rm rel}\leq 1$ and so every eigenvalue of $Q$ apart from the single zero eigenvalue is smaller than $-1$.

Writing $w = \alpha_k - \pi$ as 
\[ w = \sum_{k=0}^L \langle w, v_k \rangle_{1/\pi} v_k , \]
we have that using that $\la_0 = 0$ and $\la_i \leq - 1$ for $i \geq 1$
\[ \| w Q \|_{1/\pi}^2 = \sum_{k = 0}^L \la_k^2 \langle w, v_k \rangle_{1/\pi}^2
= \sum_{k = 1}^L \la_k^2 \langle w, v_k \rangle_{1/\pi}^2
\geq \sum_{k = 1}^L  \langle w, v_k \rangle_{1/\pi}^2
= \| w \|_{1/\pi}^2, 
\]
where we used in the last step that $\langle w, v_0 \rangle_{1/\pi} = \langle \alpha  - \pi, \pi \rangle_{1/\pi} = 0$.
Rearranging  gives us that
\begin{equation}\label{eq:diff_alpha_pi}
\|
\alpha_\kappa-
\pi
\|_{1/\pi}^2
\leq
\|
(\alpha_\kappa- \pi) Q \|_{1/\pi}^2
=
\|
\alpha_\kappa Q
\|_{1/\pi}^2.
\end{equation}
First note, that if we sum \eqref{eq_metastable} over $j$ and use that $\bar Q$ has zero row sums, we obtain that $\mathbb{E}(A_\kappa)=\rho/\lambda$, 
so that by~\eqref{eq:moment_qs}, $\rho \leq \lambda {\beta_*}$.
Next, using that $\pi_* := \min_{i \in S} \pi(i) > 0$ and the definition of $\alpha_\kappa$ in~\eqref{eq_metastable}, we have that
\[
\begin{split}
\left\| \alpha_\kappa \bar Q \right\|_{1/\pi}^2 &\leq \frac{1}{\pi_*}
\sum_{j=0}^L (\alpha_\kappa \bar Q (j))^2 
=\frac{1}{\pi_*} \sum_{j=0}^L \left(\frac{\lambda j - \rho}{\kappa}\right)^2 \alpha_\kappa(j)^2\\
& \leq \frac{\lambda^2}{\pi_* \kappa^2} \sum_{j=0}^L j^2 \alpha_\kappa(j)^2 + \frac{\rho^2}{\kappa^2} \sum_{j=0}^L \alpha_\kappa(j)^2\\
&\leq \frac{\lambda^2}{\pi_* \kappa^2} \mathbb{E}(A_\kappa^2) + \frac{\rho^2}{\kappa^2}
\leq \frac{\lambda^2 }{\pi_* \kappa^2}\left( \frac{4}{3}{\beta_*} + 2{\beta_*}^2 \right) ,
\end{split}
\]
where we used~\eqref{eq:moment_qs} in the last step.
As the right hand side tends to $0$ as $\kappa \ra \infty$, it follows from~\eqref{eq:diff_alpha_pi} that 
$\| \alpha_\kappa - \pi \|_{1/\pi} \ra 0$, which immediately implies that
$\alpha_\kappa(i) \ra \pi(i)$. 
\smallskip

\emph{Step 3: The coupling argument.}
Given $\eps > 0$, Step 2 implies that the total variation distance
between $\de^+(0)$ and $\alpha_\kappa$ is $\leq \eps /2 $. 
In particular, there is a coupling of $\de^+(0)$ and a random variable $X$ with distribution $\alpha_\kappa$ such that $\p(\de^+(0)=X)\ge1 - \eps / 2$.
As discussed in Section~\ref{metastability_section}, 
when started in $\alpha_\kappa$, we have that $T$ exactly 
has a ${\rm Exp}(\rho)$ distribution  when the Markov chain is started  in $\alpha_\kappa$. 
\smallskip

Note $\rho = \lambda \E(A_\kappa) \ra \lambda \E(\Pi)$ as $\kappa \ra \infty$
and $\E(\Pi) \ra {\beta_*}$ as $L \ra \infty$.
Therefore, for $\kappa$ and $L$  sufficiently large,
 we can extend the coupling so that there is a random variable $X$ with ${\rm Exp}({\beta_*})$ distribution such that $T=X$ with probability at
least $1-\eps$.
\end{proof}

This result on the infection rate of an updating star is the main part of proving Theorem \ref{nonadaptive_submart}, it now  remains to 
find these stars in the finite network and approximate their dynamics with those described in Proposition~\ref{lemma_dynamic_star}
\smallskip
\pagebreak[3]

For a lower bound on the set of vertices that have ever been infected, we consider a two type branching process $(A_t,D_t)_{t \geq 0}$ depending on positive parameters $\beta_*$,  $\kappa$ and $\eps$, where $A_t$ denotes the number of \emph{active} particles and 
$D_t$ the number of \emph{dormant} particles.
All particles die at a  rate of $1$ without any replacement. 
Moreover, 
at rate $\lambda {\beta_*}$ 
active particles die and are replaced with two
new dormant particles. 
Further, at rate $\kappa(1-\epsilon)$ dormant particles  die and are replaced by an active particle.
The initial condition is chosen so that $(A_0,D_0) = (1,0)$ with probability $1-\epsilon$ and $(A_0,D_0) = (0,1)$
with probability $\epsilon$.
 
\begin{table}[!h]
\centering
\begin{tabular}{l|l}
\multicolumn{1}{c}{Event} & \multicolumn{1}{c}{Rate} \\ \hline
\noalign{\smallskip}
           \mycirc{A} \raisebox{4pt}{$\rightarrow$} \mycirc{D} \mycirc{D} &   $\lambda {\beta_*}$                   \\
           \mycirc{D} \raisebox{4pt}{$\rightarrow$} \mycirc{A}            &              $\kappa(1-\epsilon)$        \\
              \mycirc{A} \raisebox{4pt}{$\rightarrow$}        &              $1$        \\
           \mycirc{D} \raisebox{4pt}{$\rightarrow$}           &               $1$      
\end{tabular}
\caption{The four rates defining the branching process $(A_t,D_t)_{t \geq 0}$.}\label{table_branching_rates}
\end{table}
The intuition is that active particle are those for which we could successfully couple the neighbourhood dynamics 
with the dynamics in Proposition~\ref{lemma_dynamic_star}, so that they effectively infect their dynamic neighbours at \emph{total} rate 
$\la \beta_*$.

\begin{proposition}\label{prop:coupling_branching} Fix $\lambda, \beta > 0$, and $\eps, \delta \in (0,1)$, then choose  $\kappa$ and $L$ large enough. Then, for $N$ sufficiently large, we can couple the contact process on the non-adaptive network and the
branching process $(A_t,D_t)_{t \geq 0}$ as described  above with parameter
\[
{\beta_*} =
\beta(1-\delta)
\]
such that 
for all $t$ with
\begin{equation}\label{eq:cond1}
(1+ L)|I_t^\ssup{N}| \leq \delta N/2
\end{equation}
we have
\[ |I_t^\ssup{N}| \geq A_t + D_t. \]
\end{proposition}

Before we prove the proposition we state a coupling lemma that we will need in the proof.

\begin{lemma}\label{prop:domination_calc} For $L \geq 1$, $\beta>0$, $\delta \in (0,1)$ and $N$ sufficiently large,
\[
P_L^{\beta(1-\delta)} \preceq \operatorname{Bin} \Big(
\Big\lfloor \Big(1-\frac{\delta}{2}\Big)N \Big\rfloor , \frac{\beta}{N} 
\Big)\wedge L.
\]
\end{lemma}

\begin{proof}
Write $B$ for a random variable with the same distribution as $\operatorname{Bin} (\lfloor (1-\frac{\delta}{2})N\rfloor , \frac{\beta}{N})$  and~$P^\alpha$ for a random variable with a  Poisson distribution with mean $\alpha$. Note that, as $N \rightarrow \infty$,
\[
B \stackrel{({\rm d})}{\longrightarrow} P^{\beta \left(1-\frac{\delta}{2}\right)}.
\]
Therefore,  for fixed $L$ and $N$ large enough
\[
\max_{i \in [L]}\left|
\mathbb{P}\big(
B \geq i \big)
-
\mathbb{P}\big(
P^{\beta (1-\frac{\delta}{2})} \geq i
\big)
\right|
<
\min_{i \in [L]}\Big(
\mathbb{P}\big(
P^{\beta (1-\frac{\delta}{2})} \geq i
\big)
-
\mathbb{P}\big(
P^{\beta (1-\delta)} \geq i
\big)
\Big),
\]
which by the triangle inequality gives $B \wedge L \succeq P^{\beta \left(1-\delta\right)} \wedge L$. 
Coupling $P^{\beta \left(1-\delta\right)} = P_L^{_{\beta(1-\delta)}}$ on the event $\{P^{\beta \left(1-\delta\right)}\leq L\}$, and using  \smash{$P_L^{_{\beta(1-\delta)}}\leq L$} almost surely
on the complement, we get that $P^{\beta \left(1-\delta\right)} \wedge L \succeq P_L^{_{\beta(1-\delta)}}$, completing the proof.
\end{proof}

 \begin{proof}[Proof of Proposition~\ref{prop:coupling_branching}]
Fix $L \geq 1$. We now construct a graph/infection process  $(\underline G_t, \underline \xi_t)_{t \geq 0}$ such that 
$\underline \xi_t \subseteq \xi_t$ and $\underline G_t$ is a subgraph of $G_t$ for all times $t \geq 0$.
We will refer to the infection status of vertices in $\underline \xi_t$ as $L$-infected and $L$-healthy.
\smallskip

To construct $(\underline G_t, \underline \xi_t)_{t \geq 0}$ we define the coupling by 
modifying the exploration construction for
$(G_t,\xi_t)_{t \geq 0}$ described in Section~\ref{sec_explore_infection} and using the same underlying graphical construction. 
Analogously to the earlier construction, we will define $L$-revealed and $L$-unrevealed vertices and refer to the
set of $L$-revealed vertices at time $t$ as $R_t^L$.
As before, we will refer to the $L$-unrevealed neighbours of a $L$-revealed vertex as its children.

At time $0$, we declare the initially infected vertex $L$-infected and $L$-revealed.  
If it has $\leq L$ neighbours in $G_0$, we copy its neighbours for $\underline G_0$, otherwise we uniformly 
choose $L$
of his neighbours in $G_0$ to connect to the initially infected vertex.

In the process $(\underline G_t, \underline \xi_t)$ we only allow infections of $L$-infected vertices
to its $L$-unrevealed neighbours in $\underline G_t$. 
When an $L$-infected vertex  infects (in the underlying graphical construction) an $L$-unrevealed neighbour, then 
this vertex becomes $L$-revealed and we connect it to vertices in $G_t$ that are $L$-unrevealed and not neighbours of $L$-revealed vertices, independently with probability $\tfrac{\beta}{N}$. Moreover, if there are more than $L$ 
of those, then we only keep $L$ uniformly chosen neighbours in $\underline G_t$. 
Note that in the `exploration with the infection' of the original process, 
at this time, this vertex may have been revealed, so that we just copy the neighbours, or it may be at unrevealed so that we sample the neighbours at that point.

Note that an $L$-infected vertex recovers as before at rate $1$, but then it is not reinfected until it updates and becomes 
$L$-unrevealed.

We deal with updates in the following way:
\begin{itemize}
\item [(i)]
If an $L$-unrevealed vertex $w$ updates that is not a neighbour of an $L$-revealed vertex, then we check if it connects to at least one $L$-infected  vertex. If  it attempts to connect to more than one, we choose one vertex $v$ uniformly to connect to 
and provided this vertex has strictly less than $L$ $L$-unrevealed neighbours we connect $w$ to $v$.
\item[(ii)]
If an $L$-unrevealed vertex updates that is a neighbour
of a $L$-revealed vertex, then  we remove this connection and the vertex 
makes no new connections (until it updates again).
\item[(iii)]
If an $L$-revealed vertex updates and it is $L$-infected, we copy 
its new neighbours from the original process, but again only keep 
connections to $L$-unrevealed vertices that are not neighbours of 
$L$-revealed vertices. 
If it is not infected, then we make it $L$-unrevealed.
\end{itemize}

Note that when a vertex is $L$-infected for the first time, then 
the number of children has a $\min\{ {\rm Bin}(N - |R_t^L|(L+1), p), L\}$
distribution. As every $L$-revealed vertex was once $L$-infected and so also infected, we have 
that $R_t^L\subset I_t^N$. Therefore, if condition~\eqref{eq:cond1} holds, we have that 
${\rm Bin}(N - |R_t^L|(L+1), p) \preceq {\rm Bin}(N - \eps/2, L)$.
Moreover, by Lemma~\ref{prop:domination_calc}, we can thus stochastically lower bound the number of children 
by $P_L^{\beta(1-\delta)}$.
\smallskip

Note that the rate that  any of the unrevealed vertices that are not a neighbour of a revealed vertex update is
$\geq \kappa (N - |R_t| (1+L)) \geq \kappa (1- \delta/2) N$. 
Also, the probability that one  of these vertices connects to a vertex $v \in \underline \xi_t$ is
with $p = \frac \beta N$
\[ \sum_{k=1}^{|\underline\xi_t|} \frac{1}{k} p^k (1- p)^{|\underline\xi_t| - k} { |R_t^L| - 1 \choose k-1 } 
= \frac{1}{|\underline\xi_t|} (1 - (1-p)^{|\underline\xi_t|}) 
\geq \frac{1-e^{-p |\underline\xi_t|}}{|\underline\xi_t|}
\geq p(1-p |\underline\xi_t|).
\]
Therefore, using that $|\underline \xi_t | \leq \frac{\delta}{2L+2} N $ on the event that~\eqref{eq:cond1} holds, the total rate at which an $L$-infected vertex acquires new neighbours 
through updating is lower bounded by 
\[ \kappa (1- \delta/2) N p(1-p |\underline\xi_t|) 
\geq \kappa \beta (1 - \delta/2) \Big( 1 -\frac{\delta}{2L+2} \beta\Big) 
\geq \kappa\beta(1 - \delta) = \kappa \beta_*, \]
as long as $L$ is chosen sufficiently large. 
\smallskip

These calculations show  that as long as~\eqref{eq:cond1} holds the first infection time when the starting vertex infects one of its children
can be stochastically lower bounded by the first hitting time of~$\dagger$ in the 
dynamics described in~\eqref{eq_degree_dynamic}.
We choose $\kappa$ and $L$ as in Proposition~\ref{lemma_dynamic_star} to then declare the initial vertex active if the coupling in Proposition~\ref{lemma_dynamic_star}
is successful, which happens with probability $1-\eps$.
Otherwise, we declare the initial vertex dormant.
\smallskip

By the coupling in Proposition~\ref{lemma_dynamic_star}, we can assume for a lower bound  that the infection by an active vertex of a child happens at rate $\lambda {\beta_*}$.  Then, both infected vertices are made dormant. 
We ignore dormant infected vertices until the time of their next vertex update when they again become active (unless they fail the coupling coupling of Proposition~\ref{lemma_dynamic_star} with probability $\epsilon$, as we see reflected in the rate of Table \ref{table_branching_rates}). 
Note that any recovery of an active or dormant vertex becomes a permanent recovery in the lower bound process.

Moreover,  the restriction that $L$-infected vertices only infect $L$-unrevealed neighbours guarantees 
the neighbourhood evolution for different active vertices is independent, 
so we can repeat the above dynamics for every new active / dormant particle.
The resulting dynamics of the total number of active / dormant particles are those given by the above branching process. 
Moreover, we clearly have that all active and dormant particles are $L$-infected vertices, 
so that we can then lower bound also on the set $I_t$ of particles that have ever been infected
as long as condition~\eqref{eq:cond1} holds.
\end{proof}

We are now ready to prove our main result of this section.

\pagebreak[3]

\begin{proof}[Proof of Theorem~\ref{nonadaptive_submart}]
The branching process fits into the framework of a standard multitype branching process. Using the notation as in~\cite{janson2004functional}, where particles of type $i$ die at rate $a_i$ and 
 then are replaced by $\xi_{ij} + \1_{i = j}$ particles of type $j$, we take
$a_1 = (1+ \lambda\beta(1-\delta))$, 
\[ \xi_{1,\cdot} = \left\{ \begin{array}{ll}  (-1,0) & \mbox{ with probability\ } \frac{1}{1+ \lambda\beta(1-\delta)}, 
\\
(-1,2) &  \mbox{ with probability\ } \frac{\lambda\beta(1-\delta)}{1+ \lambda\beta(1-\delta)}, 
\end{array} \right. \] 
and $a_2  = (1+ \kappa(1-\epsilon))$ and 
\[ \xi_{2,\cdot} = \left\{ \begin{array}{ll} (0,-1) &\mbox{ with probability\ } \frac{1}{1+\kappa(1-\epsilon)} ,
\\  (1,-1) & \mbox{ with probability\ } \frac{\kappa(1-\epsilon)}{1+\kappa(1-\epsilon)} . \end{array} \right.  \]
To determine whether the branching process survives with positive probability, 
we need to calculate the largest  eigenvalue of the  matrix
\[ A = (a_j \E[ \xi_{ji}] )_{i,j \in \{1,2\}}  = \left( \begin{array}{cc} - 1-\lambda\beta(1-\delta) & \kappa(1-\epsilon) \\ 2\lambda\beta(1-\delta)   & -1-\kappa(1-\epsilon))  \end{array} \right)  . \]
In particular, we can calculate the largest eigenvalue of $A$ explicitly and for $\kappa$ large we obtain 

\[\begin{aligned}\frac{1}{2} \Big[ & \sqrt{
\kappa^2(1-\epsilon)^2
+
6\kappa(1-\epsilon)\lambda\beta (1-\delta)
+\lambda^2\beta^2 (1-\delta)^2
}
-2-\kappa(1-\epsilon) -\lambda\beta(1-\delta)\Big] \\
& =  -1+\frac{1}{2}
\left(-\lambda\beta(1-\delta)-\kappa(1-\epsilon)\left(1-
\sqrt{
1+
\frac{6\lambda\beta (1-\delta)}{\kappa(1-\epsilon)}
+
o_\kappa\left( \frac{1}{\kappa^2} \right)
}
\right)
\right)\\
&=-1+\lambda\beta (1-\delta)+ o_\kappa(1).
\end{aligned} \]

Hence, if we assume that $\la \beta > 1$, we have that  $\la \beta(1-\delta) > 1$ for $\delta$ small enough and then in that case this eigenvalue is positive for $\kappa$ sufficiently large.
Hence, the two-type branching process survives forever with positive probability 
\[
\mathbb{P}\left(
A_t+D_t \rightarrow \infty
\right)>0.
\]

Given any $\epsilon \in (0,1)$, Proposition~\ref{prop:coupling_branching} gives $\kappa_0$ and $L_0$ above which for large $N$ we can couple the branching process and infection set:
\[
A_t+D_t \leq \left| I_t^\ssup{N}\right|,
\text{ as long as }
(1+ L)|I_t^\ssup{N}| \leq \delta N/2.
\]

Of course $| I_t^\ssup{N}|\leq N$, so on the event that the branching process tends to infinity we necessarily violate the above condition and we observe the epidemic event $\big\{I_\infty^\ssup{N}\geq \tfrac{\delta}{2L+2} N\big\}$.
\end{proof}

\subsection{Subcriticality for the non-adaptive dynamics}\label{sec:subcritical_non-adaptive}

In this section we prove Theorem \ref{nonadaptive_subcriticality_theorem} by showing that if $\la \beta < 1$ and $\kappa$ is large enough, then the contact process is subcritical.
The mean-field model of \cite{jacob2017contact} $Y = (Y_t)_{t \geq 0}$ associates to each vertex one of three states $Y_t(v) \in \{0,1,2\}$. Here, the states are interpreted as
follows: $0$ means \emph{healthy}, $2$ means \emph{infected} and the intermediate state $1$ means \emph{ready to recover}. 
\smallskip

This model is mean-field (so there is no graph structure) and 
every infected or ready-to-recover vertex $x$ infects every other vertex $y$ at rate $\lambda \beta/N$.
After an infection  event both involved vertices $x$ and $y$ become fully infected (and so switch to state $2$).
Every infected vertex updates at rate $\kappa$, which means it becomes ready to recover, and every ready to recover vertex recovers at rate $1$ and becomes healthy. 
For a more formal description, see~\cite[Section 6.1]{jacob2017contact}.
In~\cite[Proposition 6.1]{jacob2017contact} it is shown that both the 
evolving graph and the  contact process can be coupled with the process $(Y_t)_{t \geq 0}$ such that if $\mathbbm{1}_{\xi_0} \leq Y_0$, then 
$\mathbbm{1}_{\xi_t} \leq Y_t$ for all times $t > 0$.% 
\smallskip

To obtain a bound on the expected number of infected sites, also as in~\cite{jacob2017contact}, we use a supermartingale argument for the mean-field model. The difference is that on a homogeneous graph
the supermartingale is easier to define. 
But, as we are not only looking at taking the infection parameter $\lambda$ to zero, we optimise the parameters in the 
supermartingale differently.

\begin{proof}[Proof of Theorem \ref{nonadaptive_subcriticality_theorem}]

For a parameter $C \geq 1$ to be chosen later, consider the following process
\[
M(t):=\#\{v:Y_t(v)=1\}
+C\#\{v:Y_t(v)=2\}, \quad t \geq 0.
\]
By the definition of $Y$ we can bound
\begin{align*}
\frac{1}{{\rm d}t}\mathbb{E} & \big(M(t+{\rm d}t) -M(t)\big| \sF_t \big)\\
& = \sum_{v \in [N]} \1_{\{ Y_t(v) = 2\}} \kappa (1  - C) + \sum_{v} \1_{\{ Y_t(v) = 1\}}
\Big( -1  + \sum_{w \neq v} \la \frac{\beta}{N} (C -1) \Big)\\
& \qquad + \sum_{ v \in [N] } \1_{\{ Y_t(v) = 0  \}} \sum_{w \neq v} \1_{\{ Y_t(w) \geq 1 \}} \la \frac{\beta}{N} C \\
& \leq
\left(\lambda \beta (2C-1)-1\right)\#\{v:Y_t(v)=1\}+
\left(\lambda \beta C +  \kappa (1-C) \right) \#\{v:Y_t(v)=2\} .
\end{align*}
In order to write this as the smallest possible multiple of $M(t)$
we choose $C$ such that 
\[
\lambda \beta (2C-1)-1
=
\frac{\lambda \beta C +  \kappa (1-C)}{C}.
\]
Solving for $C$ gives
\[
C=\frac{1-\kappa+2\lambda \beta + \sqrt{(1-\kappa)^2+4\lambda \beta (1+\kappa+\lambda\beta)}}{4\lambda \beta}. 
\]
Note this choice always gives $C \geq 1$.
Substituting back into the previous estimate, we obtain the bound
$
\mathbb{E}\left(M(t+{\rm d}t)-M(t)\big| \sF_t \right)\leq -\epsilon M(t){\rm d}t
$
with \[ \epsilon = \tfrac12 \sqrt{ (1-\kappa)^2 + 4 \la \beta (1 + \kappa + \la \beta)} - \tfrac 12 (1+\kappa) . \] 
Therefore, $\epsilon > 0$  if and only if $\lambda \beta <1$ and
\[
\kappa>\frac{\lambda \beta (1+\lambda \beta)}{1-\lambda \beta}.
\]
Given parameters satisfying these conditions, we now use that  $(M(t)e^{\epsilon t})_{t \geq 0}$ is a supermartingale to deduce the claimed tightness. Indeed, let $t_k = \frac 12 \log k$. 
Then, if $E_1, \ldots, E_k$ are independent exponential random variables, we have
\[ \p(\max_{i \in [k]} E_i \leq t_k) = ( 1 - e^{-t_k})^k \leq e^{- e^{-t_k}K} = e^{-k^{1/2}} . \]

Write $E_1, \ldots, E_k$ for the recovery times of the first $k$ infected vertices (if there are at least $k$ infected vertices).
Then,
\[\begin{aligned}  \p (|I_\infty| \geq k ) 
& \leq \p \Big(|I_\infty| \geq k; 
\max_{i \in [k] } E_i > t_k \Big) + \p\Big(\max_{i \in [k] } E_i \leq t_k \Big) 
 \leq \p ( M_{t_k} \geq 1) + e^{-k^{1/2}}.
\end{aligned} \]
Here, we used in the last step
that on the event that $k$ vertices have been infected, if not all of those have recovered by time $t_k$, at least one vertex is still infected  by time $t_k$ and so,  using $C \geq 1$, we infer that
$M_{t_k} \geq 1$.
Finally, by Markov's inequality
\[ \p (M_{t_k} \geq 1) 
\leq \E(M_{t_k}) \leq e^{-\eps t_k} \ra 0, \]
Combining with the previous display gives the required result.
\end{proof}

\section{The adaptive dynamics}\label{sec_adapt}

In this section we consider the  adaptive contact process, where vertex updates are dependent on the infection and so the proof techniques used in
the independent case will no longer work.

\subsection{Supercriticality for the adaptive dynamics}

Before proving a sufficient criterion for survival on the 
random graph, we prove a corresponding result for the limiting model, 
the CPEF. The first condition gives survival on a particular tree, and the second on the forest of the forest infection set $I_\infty=\bigcup_{u \in \left( \mathbb{N} \right)^n}
I_\infty^u$.

\begin{lemma}\label{SIR_on_CPEF}
\begin{itemize}
\item[(i)] If
\[
\frac{\lambda \beta}{1+2\kappa+\lambda}
 >1,
 \]
then the CPEF defined in Definition \ref{CPEF_defn} satisfies $\mathbb{P}(|I^\emptyset_\infty|=\infty)>0$.

\item[(ii)] If 
\[\frac{\lambda \beta}{1+2\kappa+\lambda}
 >
 \sqrt{1-\frac{\kappa (\lambda+\kappa)\big( 1-e^{_{-\frac{\beta \lambda}{\kappa+\lambda}}} \big)}{(1+\kappa)(1+2\kappa+\lambda)} },
\]
then for the CPEF we  have the weaker statement $\mathbb{P}(|I_\infty|=\infty)>0$. 
\end{itemize}
\end{lemma}
\begin{proof}
(i) 
Note first that since both infections from a parent to a child vertex (at rate $\la$) and also updating (at rate $\kappa$) of the child vertex only occur when 
the parent vertex is infected, we have that the probability that the child updates before the first infection 
is $\frac{\kappa}{\kappa+\lambda}$, independently for each child. Therefore, before generating any infection events,  we can prune the tree by independently removing each child with probability $\frac{\kappa}{\kappa + \la}$. 
Thus the effective offspring distribution becomes \smash{$\operatorname{Pois}\big(\frac{\beta \lambda}{\kappa+\lambda}\big)$} and in the remaining process on the pruned tree, target vertices are infected at rate $\lambda + \kappa$ and 
vertices are only allowed to update after they have been infected.

For the lower bound on the number of vertices that have ever been infected in the SIS dynamics, 
we will consider the SIR dynamics on the above pruned tree, in which every vertex is \emph{removed} on its first update attempt after first becoming adjacent to infection
or nontrivial recovery attempt (i.e. its first recovery attempt after becoming infected). 
This is the SIR model on a Galton-Watson-$\operatorname{Pois}\big(\frac{\beta \lambda}{\kappa+\lambda}\big)$ tree, where
 infections occur at rate $\lambda+\kappa$ and each vertex is removed at rate $1+\kappa$. 
 In this model, each parent infects a child before removal with probability
\[
\frac{\lambda+\kappa}{1+\kappa+\lambda+\kappa}=\frac{\lambda+\kappa}{1+\lambda+2\kappa}.
\]
Therefore, this Galton-Watson tree is infinite with positive probability
if its mean satisfies
\begin{equation}\label{sir_mu}
\mu:=
\frac{\lambda+\kappa}{1+\lambda+2\kappa}\cdot \frac{\beta \lambda}{\lambda+\kappa}=\frac{\beta \lambda}{1+2\kappa+\lambda}> 1.
\end{equation}
Thus, under this condition $|I_\infty^\emptyset| = \infty$ with positive probability, which proves (i).

(ii) For the lower bound on $|I_\infty|$, we recall that by the construction of Section \ref{sec_explore_infection} $|I_\infty|$ can be seen as the total number of vertices
in a (meta-)Galton-Watson tree, where the number of offspring has the same distribution as the number of 
vertices in $I_\infty^\emptyset$ that update while being infected. We need to show that the expected number of such vertices  exceeds one.
\smallskip

 Again as in part (i), we prune vertices that update (while the parent is infected) before they have been infected themselves for the first time. Then, we run the SIR
 epidemics as in (i), where
 we treat every recovery when infected as a removal of the vertex. 
 We denote by $\mathcal{T}$ the Galton-Watson tree consisting of those vertices that 
 were infected by their parents before these were removed.
 The expected number offspring in $\mathcal{T}$ is $\mu$ as given by~\eqref{sir_mu}. 
 As this is the only interesting case, we may assume that $\mu < 1$.
 However,  now we distinguish between recovery and the updates that happen only when at least one neighbour is infected at that time. These removal events happen at rate $1+\kappa$: with probability $\frac{1}{1+\kappa}$
such an event corresponds to a 
recovery
and with probability $\frac{\kappa}{1+\kappa}$ such an event 
is a possible update (which is only successful if one of the neighbours is infected at that time). We can let the SIR infection run first and then afterwards decide if the removals were successful updates.

To determine which vertices recovered with an infected neighbour we orient each edge, so that edges point from the vertex that recovers first to the one that recovers later. Then, only vertices that have at least one out-going edge can update,
colour these white,  and all others black.
If we consider any vertex $v$ with at least one child, one of either $v$ or  its children must therefore be white.
Hence, the number of white vertices is lower bounded by 
\[ 
\sum_{n = 0}^\infty \sum_{v \in \mathcal{T} \, : \, |v| = 2n} \1_{\{ \deg^+(v) \geq 1\}} ,
 \]
where $|v|$ denotes the generation of a vertex in $\mathcal{T}$
and $\deg^+(v)$ denotes the number of children of vertex $v$. 
As the expected number of children of $\mathcal{T}$ in generation $k$
is $\mu^k$, for $\mu$ as in~\eqref{sir_mu}, 
this quantity has expectation 
\[ \sum_{n =0}^\infty (1-p_0) \mu^{2n} = (1-p_0)\frac{1}{1-\mu^2},  \]
where $p_0$ is the probability that a vertex in the SIR tree $\mathcal{T}$
has zero offspring.

Thus, the expected offspring number of the meta Galton-Watson tree, 
which corresponds to the expected number of infected vertices that updated 
successfully is lower bounded by
\begin{equation}\label{eq_cpef_mean}
\frac{\kappa}{1+\kappa}
\cdot
\frac{1-p_0}{1-\mu^2}
\end{equation}
and it remains to calculate when this exceeds $1$.

We know the mean offspring number $\mu$ from \eqref{sir_mu}.  
Moreover, to lower bound  $1-p_0$, we 
calculate 
the probability that a vertex $v$ in the underlying pruned tree with  $\operatorname{Pois}\big(\frac{\beta \lambda}{\kappa+\lambda}\big)$-offspring has at least one child
and that this one is is then infected by $v$.
This bound gives
\begin{equation}\label{eq_sir_zero_offspring}
1-p_0\geq
\left( 1-e^{-\tfrac{\beta \lambda}{\kappa+\lambda}} \right)\frac{\lambda+\kappa}{1+\lambda+2\kappa}.
\end{equation}
Substituting this bound into~\eqref{eq_cpef_mean}
and rearranging gives the claimed condition to guarantee that~\eqref{eq_cpef_mean}
is strictly greater than $1$. This completes the proof.
\end{proof}

Given the criteria for survival of the CPEF, we translate this statement back to a claim about the finite network to complete the proof of Theorem \ref{tree_supercriticality_thm}.

\begin{proof}[Proof of Theorem \ref{tree_supercriticality_thm}]
As in Section \ref{sec_explore_infection}, we reveal the network with the infection, keeping a set of revealed vertices $R_t\subset[N]$ such that each vertex is revealed on infection and unrevealed on an update while healthy.
As vertices can become revealed and then unrevealed again, we may not be able to control whether children are only attached to one parent.
Instead, we define a \emph{claimed} set $C_t \subseteq [N]$ of vertices which 
contains the initially infected vertex and all vertices that 
were ever adjacent to infected vertices before time $t$. Note that $R_t\subseteq C_t$.
Furthermore, define \[
\tau^\ssup{N}_{\epsilon} := \inf\{ t>0 \colon |C_t| \geq \epsilon N \}. 
\]
We will first show that with probability bounded away from zero, 
$\tau^\ssup{N}_\epsilon < \infty$. In a second step we will show that 
this also implies that we have order $N$ infected vertices in that case.

\emph{Step 1.} We claim that 
\[ \liminf_{N \rightarrow \infty} \mathbb{P}  ( \tau^\ssup{N} _\eps < \infty) > 0 .\]

To control the number of claimed vertices, we use that every vertex that has ever been infected by time $t$ is also in $C_t$. Then, we use
the SIR lower bound 
described in the proof of the previous lemma 
by coupling $I_t^\ssup{N}$ to an SIR process on an evolving Galton-Watson forest up to time $\tau_\eps^\ssup{N}$.

Before time $\tau_\eps^\ssup{N}$, at each step of the exploration when either an infected vertex 
updates or a vertex becomes newly infected, we connect this vertex
to ${\rm Bin} (N- |C_t|, \beta/N)$ unclaimed vertices. 
Thus for  a lower bound on the number of infected vertices up to time
$\tau_\eps^\ssup{N}$, we can couple to a version of the CPEF,
but with 
$\operatorname{Bin}(\lfloor (1-\eps) N \rfloor,\beta/N)$
distributed offspring.
As in the proof of Lemma~\ref{SIR_on_CPEF}, 
we can prune the trees, by only keeping children if they are infected before they are updating
(which happens with probability $\frac{\lambda}{\lambda + \kappa}$), 
so that this becomes 
an SIR process on an evolving $\operatorname{Bin}(\lfloor (1-\eps) N \rfloor,\frac{\lambda\beta}{(\lambda + \kappa)N})$-Galton-Watson forest. 
Given $L$ and $N$ sufficiently large, by Lemma~\ref{prop:domination_calc}, the  
offspring number is stochastically
lower bounded by a random variable~$D_{L,\epsilon}$ that is $\operatorname{Pois}( \tfrac{\beta \lambda}{\kappa+\lambda}(1-2\epsilon))$-distributed.
Thus, by the same as argument before, the  SIR on an evolving process survives with positive probability 
if 
\begin{equation}\label{eq:approx_survival} \frac{\kappa}{1+\kappa}
\cdot
\frac{1-p_0^{L,\eps}}{1-\mu_{L,\eps}^2} > 1, 
\end{equation}
where the analogous quantities to the previous argument are now
\[
\mu_{L,\epsilon}=\frac{\mathbb{E}(D_{L,\epsilon}) ( \lambda+\kappa)}{1+2\kappa+\lambda} ,
\quad \mbox{ and }\quad
1 - p_0^{\eps, L} \geq (1-  \mathbb{P} (D_{L,\eps} = 0) )\frac{\lambda + \kappa }{1 + \lambda + 2\kappa}.\]
As $L \rightarrow \infty $ and $\epsilon \downarrow 0$, 
$\mu_\eps \ra \mu$, for $\mu$ as in~\eqref{sir_mu} and 
$\mathbb{P}(D_{L,\eps} = 0) \ra e^{-\tfrac{\beta \lambda}{\kappa+\lambda}}$, 
Thus, our assumption~\eqref{eq:suff_supercritical} implies that for $\eps$ small
enough and $K$ large enough, the condition~\eqref{eq:approx_survival} holds
and the number of vertices $|I^{\eps,L}_\infty|$ that have ever been infected in the 
approximating SIR process on an evolving forest is infinite with positive probability. 
Therefore, as $|C_t| \leq N$, we have that 
\[ \limsup_{N \rightarrow \infty} \mathbb{P}  ( \tau^\ssup{N} _\eps = \infty) 
\leq \mathbb{P} (|I^{\eps,L}_\infty | < \infty) < 1,  \]
which gives our claim.

 \emph{Step 2:} We now show that having $\epsilon N$ claimed vertices also implies that there are of order $N$ infected vertices.

Let $s_1 < s_2 \ldots < s_K$ be the consecutive times when  unclaimed vertices are being added to the set of claimed vertices. 
Let $v_k$ be the vertex with smallest label that is added to the claimed set
at time $s_k$. Note that an unclaimed vertex cannot update and also that
if a claimed but unrevealed vertex updates, no new claimed vertices are added by construction. 
Therefore, a new vertex can only be claimed if a claimed and revealed vertex becomes infected 
or updates. Let $w_k$ be the vertex whose infection or updating 
leads to the claim of $v_k$.
If we let  $N_k^R$ be the number of unclaimed vertices that are
being claimed at time $s_k$, then
\[
N_k^R \preceq 
B_k
\stackrel{({\rm d})}{=}
\operatorname{Bin}\big(
N-1, \tfrac{\beta}{N}
\big)
\]
Further, we can generate $(B_k)_{k \in \mathbb{N}}$ as an i.i.d.\ sequence in advance and thin down the binomial trials as required.

Also, for each $k \in [K]$, we can define the indicator $J_k$
of the event that $w_k$ infects $v_k$ successfully before 
the recovery of $w_k$ or an update of either vertex. 
Thus, 
conditionally on $\{J_j : j<k\}$ and $\{\xi_{t}: t\leq s_k\}$, we see that
the event $\{ J_k = 1\}$ has  probability $\frac{\lambda}{1+\lambda+2\kappa}$.
Define the error events
\[
\mathfrak{E}_1
=
\Bigg\{
\sum_{k=1}^{\big\lfloor\tfrac{2\epsilon N}{3 \beta}\big\rfloor} B_k > \epsilon N
\Bigg\}\quad \mbox{and }
\quad
\mathfrak{E}_2
=
\Bigg\{
\sum_{k=1}^{\left\lfloor\tfrac{2\epsilon N}{3 \beta}\right\rfloor} J_k <
\frac{1}{4}
\frac{\lambda}{1+\lambda+2\kappa} \frac{\epsilon N}{ \beta}
\Bigg\}.
\]
Using a standard Chernoff bound 
$
\mathbb{P}\left( \left| \operatorname{Bin}(
n, p 
) -np \right| > \frac{np}{2}
\right)<2 e^{-np/12}
$
we find
\begin{equation}\label{eq:error_1}
\begin{split}
\mathbb{P}\left(
\mathfrak{E}_1
\right)
&=
\mathbb{P}\left(
\operatorname{Bin}\left(
\left\lfloor\frac{2\epsilon N}{3 \beta}\right\rfloor(N-1), \tfrac{\beta}{N}
\right)>\epsilon N
\right)\\
&\leq
\mathbb{P}\left(
\operatorname{Bin}\left(
\left\lfloor\frac{2\epsilon N}{3 \beta}\right\rfloor N
, \tfrac{\beta}{N}
\right)>\epsilon N
\right)
\leq 2 e^{ - \beta \lfloor\frac{2\epsilon N}{3 \beta}\rfloor / 12 }
\end{split}
\end{equation}
and by the same bound
\begin{equation}\label{eq:error_2}
\begin{split}
\mathbb{P}\left(
\mathfrak{E}_2
\right)
&=
\mathbb{P}\left(
\operatorname{Bin}\left(
\left\lfloor\frac{2\epsilon N}{3 \beta}\right\rfloor, \frac{\lambda}{1+\lambda+2\kappa}
\right)<
\frac{1}{4}
\frac{\lambda}{1+\lambda+2\kappa} \frac{\epsilon N}{ \beta}
\right)\\
&\leq
\mathbb{P}\left(
\operatorname{Bin}\left(
\left\lfloor\frac{2\epsilon N}{3 \beta}\right\rfloor, \frac{\lambda}{1+\lambda+2\kappa}
\right)<
\frac{1}{2}
\frac{\lambda}{1+\lambda+2\kappa} \left\lfloor\frac{2\epsilon N}{3 \beta}\right\rfloor
\right)\\
&\leq
2 \exp \left(
-\frac{1}{12}\frac{\lambda}{1+\lambda+2\kappa} \left\lfloor\frac{2\epsilon N}{3 \beta}\right\rfloor
\right).
\end{split}
\end{equation}

We now claim that on $\mathfrak{E}_1^c\cap \mathfrak{E}_2^c$, the event $\{ \tau_\eps^\ssup{N} < \infty\}$ implies that $I_\infty^\ssup{N} \geq 
\frac{1}{4}
\frac{\lambda}{1+\lambda+2\kappa} \frac{\epsilon}{ \beta}N$. 
Indeed, as $\tau_\eps^\ssup{N} < \infty$ we know that eventually at least $\eps N$ vertices become claimed. 
However, 
on the event $\mathfrak{E}_1^c$, we know that 
it takes us at least $\lfloor\frac{2\epsilon N}{3 \beta}\rfloor$
distinct events at which new claimed vertices are added
 to reach $\eps N$ claimed vertices. 
Finally, on $\mathfrak{E}_2^c$, altogether these events correspond to 
at least $\frac{1}{4}
\frac{\lambda}{1+\lambda+2\kappa} \frac{\epsilon}{ \beta}N$ distinct infected vertices,  as at each `successful' event $\{ J_i  =1\}$ an unclaimed vertex is infected
that by definition has not been infected before.
Hence, 
\[
\mathbb{P} \Big(
 I_\infty^\ssup{N}\geq \frac{1}{4}
\frac{\lambda}{1+\lambda+2\kappa} \frac{\epsilon}{ \beta}N
 \Big) \geq 
\mathbb{P}\left(
\mathfrak{E}_1^C \cap \mathfrak{E}_2^C  \cap \{\tau_\epsilon<\infty\}
\right)
\geq
\mathbb{P}\left(\tau_\epsilon<\infty\right)
-
\mathbb{P}\left(\mathfrak{E}_1\right)
-
\mathbb{P}\left(\mathfrak{E}_2\right), 
\]
and the previous estimates~\eqref{eq:error_1} and~\eqref{eq:error_2} together with Step 1 show that the 
right hand side is bounded away from zero as $N \rightarrow \infty$, 
as required.
\end{proof}

\subsection{Subcriticality for the adaptive dynamics}

In order to show subcriticality for the contact process, we compare the process to the case without updating. We do this on a host tree which is first a general tree, and we will later specialise to a Poisson Galton-Watson tree as in the CPEF.
The root-added contact process was introduced in \cite{bhamidi2019survival} to control the depth of a standard contact process, but here to control the size of the infection we use a modification which results in  connected infection set.

\begin{definition}[Subtree Contact Process]\label{def_SCP}
Construct the \emph{Subtree Contact Process} (SCP) on a rooted tree $(\mathcal{T},\mathbbm{o})$, denoted the \emph{host tree}, by adding an additional vertex $\mathbbm{o}^+$ only adjacent to the root $\mathbbm{o}$. This additional vertex is always infected and cannot recover. We further modify the contact process so that no other vertex can recover unless all of its children are healthy. 
\end{definition}

With these dynamics the set of infected vertices (without $\mathbbm{o}^+$, which is always infected) is a rooted subtree of $\mathcal{T}$ at all times. We define $\sT$ as the set of finite subtrees of $\mathcal{T}$ which are rooted at $\mathbbm{o}$, which we interpret as the state space of the SCP.

Note that the SCP has no absorbing state and is irreducible on the set $\mathcal{T}$ of subtrees. Moreover,  in any rooted subtree $T$ we can only see recoveries  at a vertex $\ell \in T$
if $\ell$ has no children in $T$.
In particular, the detailed balance equation  reduce to checking that 
for every $T \in \sT$ and $\ell$ a vertex without children in $T$.
\[
\lambda \pi(T \setminus \{\ell\}) = \pi(T),
\]
for $\pi$ a measure on $\sT$.
Thus, if $Z_\lambda:=\sum_{T \in \sT} \lambda^{|T|}<\infty$, we define for $T \in \sT$ 
\begin{equation}\label{MRACP_stationary_dist}
\pi_\lambda(T):=\frac{\lambda^{|T|}}{Z_\lambda} .
\end{equation}
Note that $\pi_\lambda$
clearly satisfies the detailed balance equation so that is the stationary distribution of the SCP and the process is reversible.

Because we made simplifications by adding infection and preventing recovery, the SCP infection stochastically dominates the usual contact process on the same host tree. We will therefore use the SCP to bound the infection sets. 

\begin{lemma}\label{Z_bound}
On the random (finite or infinite) $\operatorname{Pois}(\beta)$-Galton-Watson tree, the normalising constant $Z_\lambda$ defined in \eqref{MRACP_stationary_dist} satisfies
\[
\mathbb{E}(Z_\lambda)\leq 1+\lambda+\beta \lambda^2 + \frac{e^3 \lambda^3 \beta^2}{3 \sqrt{6 \pi }(1-\lambda \beta e)},
\]
whenever $\lambda \beta e<1$.
\end{lemma}

\begin{proof}
For any $k \in \mathbb{N}$, let $s_k$ denote the expected number of subtrees of size $k$ containing the root $\mathbbm{o}$ in a $\operatorname{Pois}(\beta)$-Galton-Watson tree and let $s_k^{_{(N)}}$ denote the expected number of subtrees of size $k$ containing the vertex $1$ in the Erd\H{o}s-R\'enyi Graph with parameter $\beta/N$.\smallskip

First we remark that $s_1^{_{(N)}}=1$ and $s_2^{_{(N)}}=\mathbb{E}(\de(1))=\beta\left(1-\frac{1}{N}\right)\leq \beta$. We define $s_0^{_{(N)}}=1$ as the SCP should have only one empty state.
For $k\geq 3$,  we see in \cite{chin2018subtrees} (from Cayley's formula) that the complete graph has $\binom{N}{k}k^{k-2}$ subtrees of size $k$ and so similarly it has $\binom{N}{k-1}k^{k-2}$ subtrees of size $k$ containing the vertex $1$. Each such tree is realised as a subgraph of the Erd\H{o}s-R\'enyi Graph with probability $(\beta/N)^{k-1}$, and hence by linearity of expectation
as $N\rightarrow \infty$
\[
\begin{split}
s_k^{_{(N)}}&=\left(\frac{\beta}{N}\right)^{k-1}\binom{N}{k-1}k^{k-2}
\rightarrow \frac{\beta^{k-1}k^{k-2}}{(k-1)!} = \frac{\beta^{k-1}k^{k-1}}{k!} \\
&\leq \frac{\beta^{k-1}k^{k-1}}{\sqrt{2 \pi k} k^k e^{-k} }
= \frac{1}{\beta \sqrt{2 \pi } k^{3/2}  } (\beta e)^k
\leq \frac{1}{3\beta \sqrt{6 \pi } } (\beta e)^k , \\
\end{split}
\]
where we used $k\ge 3$ in the last step.

Using \cite[Corollary 2.11]{van2016random2} we have local weak convergence of the Erd\H{o}s-R\'enyi graph  to a $\operatorname{Pois}(\beta)$-Galton-Watson tree
in probability, which implies local weak convergence in distribution. The number of rooted subtrees of size $k$ is not a bounded function, but the well-known portmanteau theorem still gives us
\[
\liminf_{N \rightarrow \infty} s_k^{_{(N)}} \geq s_k.
\]
Hence, 
\[
\begin{split}
\mathbb{E}(Z_\lambda)=\sum_{k=0}^\infty \lambda^k s_k &\leq
1+\lambda+\beta \lambda^2 + \frac{1}{3\beta \sqrt{6 \pi } } \sum_{k=3}^\infty (\lambda \beta e)^k\\
&=
1+\lambda+\beta \lambda^2 + \frac{e^3}{ 3\sqrt{6 \pi }  } \cdot  \frac{\lambda^3 \beta^2 }{1-\lambda \beta e}
,
\end{split}
\]
as required.
\end{proof}

Controlling the mean of $Z_\lambda$ allows us to control the mean expected recovery time for the SCP.

\begin{definition}[Recovery time $R$]
Recalling from Definition \ref{def_SCP} that the additional vertex $\mathbbm{o}^+$ is not included in the infection set $I_t$, the recovery time $R$ for the SCP is the first hitting time of the empty set infection, i.e.\ $R=\inf \{t>0 : I_t=\emptyset\}$.
\end{definition}

As a Markov chain hitting time this could make sense from any initial state, but we will typically be interested in the initial state $\{\mathbbm{o}\}$ of just the root infected. We introduce a \emph{size-slowed} version of the SCP for which bounds on the recovery time $R$ will imply bounds on the size of the infection sets.

\begin{definition}
\label{def_SSCP}
Let the \emph{Slowed Subtree Contact Process} (SSCP), for some constant $\rho>1$, follow the same dynamics as the SCP process, but leave each state $T$ at a rate decreased by the factor $\rho^{-|T|}$. Hence each valid recovery occurs at rate $\rho^{-|T|}$, each infection at rate $ \lambda\rho^{-|T|}$.
\end{definition}

Note that the stationary distribution of this process is $\pi_{\lambda \rho}$ 
as defined in~\eqref{MRACP_stationary_dist}, 
but with $\lambda$ replaced by $\lambda \rho$.
By comparing with Lemma~\ref{Z_bound}
we must have $\lambda \beta \rho e<1$ 
to have a finite 
expected recovery for the SSCP. 
Thus, while larger $\rho$ values give better control over the infection size, we pay a penalty in the permitted values of $\lambda \beta$.

\begin{lemma}\label{lemma_infection_set_time}
Let $\rho>1$. If we define
\[
\theta:=\min_{k \geq 2}  \frac{\rho^{k}}{(k-1)},
\]
then in the SSCP we have
\[
\mathbb{E}^{\lambda,{\rm slow}}_{\{\mathbbm{o}\}}
\left(
R
\right)
\geq
\rho-\theta+\theta \,
\mathbb{E}^{\lambda,{\rm slow}}_{\{\mathbbm{o}\}}
\big(
\big|
I^\emptyset_\infty
\big|
\big),
\]
where 
 $I_\infty^\emptyset$ 
 is the set of vertices that have been infected at some point as in Definition \ref{def_superandsub}. 
\end{lemma}

\begin{proof}
Note that since the root $\mathbbm{o}$ can only recover if all other vertices have recovered, 
if more than $k \geq 2$ vertices are infected, then at most $k-1$ vertices are 
able to recover. Hence, the  slowed model has a maximal recovery rate $1/\theta$
over all states which are not ${\{\mathbbm{o}\}}$. The recovery rate at ${\{\mathbbm{o}\}}$ is simply $1/\rho$ which we can consider separately.

We can therefore construct the SSCP over the time interval $[0,R]$ with the following Poissonian construction, which is not quite the graphical construction of \cite{liggett1985interacting} but is better suited to our argument: we  define a process $(\xi_t, U_t)_{t \geq 0}$, where  $(\xi_t)_{t \geq 0}$  has the same distribution as the SSCP and $U_t$ is a vertex in~$\xi_t$, which we fix at specific times and otherwise remains constant.

\begin{enumerate}[(i)]
    \item Initially set $U_0=\mathbbm{o}=:v_1$, $\xi_0=\{\mathbbm{o}\}$ and attach to $\mathbbm{o}$ the Poisson process $X_{(1)} \sim PP\big(\tfrac{1}{\rho}\big)$.
    \item \label{sscp_interaction} Then simultaneously:
    \begin{itemize}
        \item Each vertex infects each of its neighbours at rate $\lambda \rho^{-|\xi_t|}$ as usual in the SSCP. If we infect a new vertex by adding it to the infection set $I_t$ at some time $t$ with
        \[
        |I_{t-}|=k-1,
        |I_{t}|=k
        \]
        then label this vertex $I_{t} \setminus I_{t-}=\{v_k\}$, record $i(v_k)=t$, and attach to $v_k$ an independent Poisson process $X_{(k)}\sim PP\big(\tfrac{1}{\theta}\big)$;
        \item Each vertex $v_k$ attempts to recover at the times of the left-continuous inverse $C^{-1}_{(v_k)} \big(X_{(k)}\big)$, where $C_{(v_k)}$ is the clock
        \[
C_{(v_k)}(a):=\int_{i(v_k)}^a \mathbbm{1}_{U_s=v_k}{\rm d}s.
\]
For $k=1$ these attempts always correspond to true recoveries, but for $k \geq 2$ we overestimated the rate of recovery and so we correct by Poisson thinning. That is, given a recovery attempt of $v_k$ with $k\geq 2$ at time $t$:
\begin{itemize}
    \item Reject the recovery attempt if $v_k$ has any children in $\xi_t$;
    \item Also reject the recovery attempt with probability \[1-\theta(|\xi_t|-1)\rho^{-|\xi_t|}\]
    by generating an independent Bernoulli variable.
\end{itemize}
    \end{itemize}
    \item After either event of \ref{sscp_interaction} at some time $t$, if we have an empty infection set, then we have arrived at time $t=R$ and can stop the process. Otherwise, regardless of an infection being trivial or a recovery rejected, independently draw
    \[
U_t\sim
\begin{cases}
\operatorname{Uniform}\left[\xi_t\setminus {\{\mathbbm{o}\}} \right] & \mbox{if }\xi_t\neq {\{\mathbbm{o}\}}\\
\mathbbm{o} & \mbox{if }\xi_t = {\{\mathbbm{o}\}}
\end{cases}
\]
and return to step \ref{sscp_interaction}. 
\end{enumerate}

This construction works because while $\xi_t \neq \{\mathbbm{o}\}$, if we average over $U_t \sim \operatorname{Uniform}\left[\xi_t\setminus {\{\mathbbm{o}\}} \right]$, the next potential recovery 
occurs at rate $\frac{1}{(|\xi_t|-1) \theta}$ per vertex $v \in \xi_t \setminus {\{\mathbbm{o}\}}$. 
Then if $v$ has no infected children in $\xi_t$, we can multiply by the explicit thinning probability above to see that the effective rate of a successful recovery is
\[ \frac{1}{(|\xi_t|-1) \theta} \cdot \theta(|\xi_t|-1)\rho^{-|\xi_t|} =  \rho^{-|\xi_t|}, \]
as in the original SSCP process.

The power of this construction is that a vertex $v$ is only allowed to update when $U_t = v$ and $U_t$
only chooses one vertex at a time.
Moreover, in order for $\xi_t$ to come back to the initial state, all vertices that have been infected must have recovered.
For each vertex $j$ at least $\min X_{(j)}$ time must have passed before recovery and the time windows when each vertex is allowed to recover are disjoint by construction. 
Therefore
\[
R \geq \sum_{j=1}^{|I_\infty|} \min X_{(j)}
\]
and by taking expectations and using Fubini's theorem we get
\[
\begin{split}
\mathbb{E}_{\{\mathbbm{o}\}}^{\lambda, {\rm slow}}(R)
&\geq
\sum_{j=1}^\infty
\mathbb{E}_{\{\mathbbm{o}\}}^{\lambda, {\rm slow}}\left(
\mathbbm{1}_{|I^\emptyset_\infty|\geq j}
\min X_{(j)}
\right) \\  &
=
\sum_{j=1}^\infty
\mathbb{P}_{\{\mathbbm{o}\}}^{\lambda, {\rm slow}}\left(
|I^\emptyset_\infty|\geq j
\right)
\mathbb{E}_{\{\mathbbm{o}\}}^{\lambda, {\rm slow}}\left(
\min X_{(j)}
\right)\\
&=\rho+\theta
\sum_{j=2}^\infty
\mathbb{P}_{\{\mathbbm{o}\}}^{\lambda, {\rm slow}}\left(
|I^\emptyset_\infty|\geq j
\right)
=\rho+\theta
\left(
\mathbb{E}_{\{\mathbbm{o}\}}^{\lambda, {\rm slow}}\left(
|I^\emptyset_\infty|\right)-1
\right).
\end{split}
\]

Here we used in the second equality that the event that at least $j$ vertices get infected is independent 
of $X_{(j)}$, as the latter is only used in the construction of the process after the $j$th vertex is infected.
\end{proof}

\begin{remark}
The expected recovery time of the SCP diverges as $\lambda \beta \uparrow 1/e$, so we want to allow parameters as close to this limit as possible without going so close that the expected infection set is not small enough. By some numerical computations that we will not show, we calculate that the following assumption is very close to optimal for maximising the region of $\lambda \beta$ permitted in Lemma \ref{lemma_CPEF_subcrit}, and in the resulting main theorem (Theorem \ref{subcriticality_theorem}).
\end{remark}

\begin{lemma}\label{MRCP_infection}
If $\lambda \beta e \leq \frac{3}{4}$ then 
the number of vertices $|I_\infty^\emptyset|$ that have ever been infected in the root tree $T^\emptyset$ of the CPEF satisfies
\[
\mathbb{E}(|I_\infty^\emptyset|)\leq
1+
\frac{27 \beta \lambda}{16} \Big(
1 + \sqrt{\frac{2}{ 3\pi }} \cdot  \frac{ 2e^3\beta \lambda }{9-12 \beta \lambda e}
\Big).
\]
\end{lemma}

\begin{remark}
In fact, we prove the  bound for the $\kappa=0$ model, the contact process on a static Galton-Watson tree, and then use monotonicity.
\end{remark}

\begin{proof}
To upper bound the infection it is sufficient to look at the model completely neglecting updating on each local tree set  -- so, the host tree $\cT$ is Galton Watson with offspring distribution $\operatorname{Pois}(\beta)$ as in Definition \ref{CPEF_defn}. Of course, locally, updating is simply deletion of a vertex -- therefore on a local tree the infection with no updating dominates that with updating.

To analyse the spread of an infection without updating, we can look at the SSCP  as an upper bound. For some 
$\rho\in \big(
1,
\tfrac{1}{\beta \lambda e}
\big)
$
to be chosen later,
we recall (from the detailed balance equations) that the distribution \smash{$\pi_{\lambda\rho}(T)=(\lambda\rho)^{|T|}/Z_{\lambda\rho}$} of Equation \eqref{MRACP_stationary_dist} is the stationary distribution, which will allow us to bound recovery times via Lemma \ref{Z_bound}.

Write $\mathbb{E}^{\lambda,{\rm slow}}$ for the measure of the SSCP with infection rate~$\lambda$ on the Poisson Galton-Watson host tree $\cT$.
By Kac's formula \cite[Theorem 3.5.3]{norris1998markov} applied to the SSCP we have for the recovery time~$R$,
\[
\mathbb{E}^{\lambda,{\rm slow}}_{\{\mathbbm{o}\}}\left(R \big|\cT\right)
=
\mathbb{E}^{\lambda,{\rm slow}}_\emptyset\left(R \big|\cT\right)-\frac{1}{\lambda}
=
\frac{1}{\lambda \pi_{\lambda \rho}\left( \emptyset \right)}-\frac{1}{\lambda}
=
\frac{Z_{\lambda \rho}-1}{\lambda},
\]
because there is no slowing in the empty tree state. We recall Lemma \ref{Z_bound} and average over $\cT$ to see
\begin{equation}\label{eq_recovery_mean}
\mathbb{E}^{\lambda,{\rm slow}}_{\{\mathbbm{o}\}}\left(R \right)
\leq
\rho+\beta \lambda \rho^2 + \frac{e^3}{ 3\sqrt{6 \pi }  } \cdot  \frac{\lambda^2 \rho^3 \beta^2 }{1-\lambda \rho \beta e}.
\end{equation}

Note that by assumption 
$
(1,
\tfrac{4}{3}
)\subseteq (1,\frac{1}{\beta \lambda e})$, so we
can take $\rho \in 
\left(
\frac{5}{4},
\frac{4}{3}
\right)$.
Thus, when applying Lemma \ref{lemma_infection_set_time} we have
$\theta= {\rho^5}/{4}$
and hence
\[
\mathbb{E}^{\lambda,{\rm slow}}_{\{\mathbbm{o}\}}\left( R  \right)
\geq 
\rho + \frac{\rho^5}{4} \left( \mathbb{E}^{\lambda,{\rm slow}}_{\{\mathbbm{o}\}}\big(|I^\emptyset_\infty|\big) -1 \right) .
\]
By combining with \eqref{eq_recovery_mean}, we get
\[
\rho + \frac{\rho^5}{4} \left( \mathbb{E}^{\lambda,{\rm slow}}_{\{\mathbbm{o}\}}\big(|I^\emptyset_\infty|\big) -1 \right)
\leq
\mathbb{E}^{\lambda,{\rm slow}}_{\{\mathbbm{o}\}}\left( R \right)
\leq
\rho+\beta \lambda \rho^2 + \frac{e^3}{ 3\sqrt{6 \pi }  } \cdot  \frac{\lambda^2 \rho^3 \beta^2 }{1-\lambda \rho \beta e} .
\]
Rearranging, we obtain
\[
\begin{split}
\mathbb{E}^{\lambda,{\rm slow}}_{\{\mathbbm{o}\}}\big(|I^\emptyset_\infty|\big)
&\leq
1+
\frac{4\beta \lambda}{\rho^3} \left(
1 + \frac{e^2}{ 3\sqrt{6 \pi }  } \cdot  \frac{\lambda \rho \beta e }{1-\lambda \rho \beta e}
\right).
\end{split}
\]

Finally we take $\rho \uparrow \frac{4}{3}$ which obtains the claimed result, because we can couple the SSCP and SCP with rate $\lambda$ such that they follow the same paths until recovery and so have the same infection set.
\end{proof}

We now return to the study of the CPEF. 
From the bound above on the  spread of the infection within a tree, we must bound the number of infected vertices which update while infected and thus spread the infection to another tree. If the expectation of the number of these vertices is less than $1$, by 
referring to the meta-Galton-Watson tree structure described in Section~\ref{sec_explore_infection}
we can then deduce that the infection process dies out on the evolving forest.

\begin{lemma}\label{lemma_CPEF_subcrit}
If $\kappa>0$ and $\beta \lambda< 0.21$ then the CPEF has
\[
\mathbb{E}\bigg(
\sum_{n=1}^\infty
\sum_{u \in \left( \mathbb{N} \right)^n}
\mathbbm{1}_{|I_\infty^u| \neq 0}
\bigg)
<\infty
\]
i.e.\ the expected number of other local trees that become infected is finite.
\end{lemma}

\begin{proof}
No vertex can leave the root tree and start a new infected tree unless it was a member of $I_\infty^\emptyset$. Further, if $|I_\infty^\emptyset|-1$ vertices have updated then the last cannot update, as it would not have any infected neighbours remaining.
Hence we conclude that the meta-Galton-Watson process of trees is subcritical whenever
\[
\mathbb{E}\big(|I_\infty^\emptyset|-1\big)<1.
\]

By Lemma \ref{MRCP_infection}, we know this is satisfied whenever $\beta \lambda<\frac{3}{4e}=0.276\dots$ and
\[
\frac{27 \beta \lambda}{16} \bigg(
1 + \sqrt{\frac{2}{ 3\pi }} \cdot  \frac{ 2e^3\beta \lambda }{9-12 \beta \lambda e}
\bigg)<1, 
\]
which can be rewritten as quadratic inequality in $\beta \lambda$
\[
\begin{split}
&\frac{243}{16}\beta \lambda
-
\frac{81e}{4}\beta^2 \lambda^2
+
\sqrt{\frac{2}{ 3\pi }} \cdot \frac{27e^3}{8}\beta^2 \lambda^2
<
9-12 \beta \lambda e\\
\iff
&\bigg(\frac{81e}{4}-\sqrt{\frac{2}{ 3\pi }} \cdot \frac{27e^3}{8}\bigg) \beta^2 \lambda^2
-\bigg(12e+\frac{243}{16}\bigg)\beta \lambda
+9
>0.
\end{split}
\]
We solve the quadratic to find a region $\{\beta \lambda<L\}$, where the limit $L$ is
\[
L=
\frac{81+64e - \sqrt{1152 e^3 \sqrt{\frac{6}{\pi }}+4096 e^2-10368 e+6561}}{12 e \Big(18-e^2 \sqrt{\frac{6}{\pi }}\Big)}
>0.21
\]
and hence $\{\beta\lambda<0.21\}$ is sufficient for subcriticality.
\end{proof}

We now give an alternative sufficient condition for subcriticality that works 
better for large $\kappa$.

\begin{lemma}\label{le:CPEF_sub_large_kappa}
If $(2\beta-1) \lambda<\kappa$, then
 the CPEF satisfies 
\[
\mathbb{E}\bigg(
\sum_{n=1}^\infty
\sum_{u \in \left( \mathbb{N} \right)^n}
\mathbbm{1}_{|I_\infty^u| \neq 0}
\bigg)
<\infty .
\]
\end{lemma}

\begin{proof}
We percolate the (generic) root tree of the CPEF just as in the proof of Lemma \ref{SIR_on_CPEF}, removing any vertex that would update away before it is ever infected. On this percolated root tree we have offspring mean
\[
\mu=\beta\frac{\lambda}{\lambda+\kappa}
\]
and by assumption we have $\mu<\frac{1}{2}$. Hence the total size of this tree has mean
\[
\frac{1}{1-\mu}<2.
\]
However,  $I_{\infty}^\emptyset$ is a subset of this tree and no more than $\left|I_{\infty}^\emptyset\right|-1$ offspring can update while being infected themselves (as you need at least one infected neighbour to update). Therefore, the
bound on the size of the tree implies that in the meta-tree structure of the CPEF the expected number of offspring is  less than $1$.
\end{proof}

Having found a region of subcriticality on the CPEF, it remains to check that we can couple the CPEF back to the original network problem. The first ingredient is controlling how closely we can approximate the degrees as Poisson.

%\begin{lemma}[Theorem 1 of \cite{barbour1984rate}]\label{poisson_tv}
%\[
%d_{TV}\Big(
%\operatorname{Bin}\big(N, \tfrac{\beta}{N}\big),
%\operatorname{Pois}(\beta)
%\Big)
%\leq
%\frac{\beta \big( 1- e^{-\beta} \big) }{N}
%\leq
%\frac{\beta }{N}
%\]
%and hence we can couple two realizations of the two distributions so that they are identical with probability $1-\beta/N$.
%\end{lemma}

\begin{proof}[Proof of Theorem \ref{subcriticality_theorem}]
First fix some large constant $B \in \mathbb{N}$ and observe that by Lemmas~\ref{lemma_CPEF_subcrit} and~\ref{le:CPEF_sub_large_kappa}  
in expectation we  see only finitely many tree sets $T^u$. Hence by Markov's inequality we will see fewer than $B$ tree sets with high probability in $B$. By Lemma \ref{MRCP_infection}, respectively the proof of Lemma~\ref{le:CPEF_sub_large_kappa},
in each tree the expected number of vertices ever to become infected 
is bounded by $2$. Hence, 
in the first $B$ (i.i.d.) local trees of the CPEF 
the expected total infection size is bounded by $2B$.
%Therefore, with high probability the total number of vertices that are ever infected is bounded by~$B^2$, with probability $1-\nicefrac{2}{B}$.
Therefore by Markov's inequality again these $B$ tree sets comprise at most~$B^2$ vertices that are ever infected, with probability $1-\nicefrac{2}{B}$.\smallskip

We now construct a coupling of exploration of the model described in Section~\ref{sec_explore_infection} with the CPEF.
In the following, we will label vertices in the CPEF and the coupling  works up until the first repeated label. 
In the CPEF, the $i$th sampled offspring degree is an i.i.d. $p_i \sim \operatorname{Pois}(\beta)$. Attach an i.i.d. label $u_{j+1}\sim \operatorname{Uniform}([N])$ to the $j$th half-edge constructed in this way.\smallskip

By Poisson thinning, for each  $k \in [N]$, the number of offspring $C_i(k)$
with label $k$ is independently $\operatorname{Pois}(\nicefrac{\beta}{N})$
distributed.
%this can equivalently be seen as revealing an i.i.d. $\operatorname{Pois}(\nicefrac{\beta}{N})$ number of edges to each other vertex in $[N]$.
%We calculate that
By a standard coupling result, see e.g.~the proof of Thm.\ 2.10 in~\cite{van2016random},
\[
\de_{\rm TV}\left(\operatorname{Pois}\left(\frac{\beta}{N}\right),\operatorname{Ber}\left(\frac{\beta}{N}\right) \right) 
\leq
\frac{\beta^2}{N^2}. %+O\left(\frac{1}{N^3}\right)
\]

Therefore, we can couple $C_i(k)$ with a 
$\operatorname{Ber}(\nicefrac{\beta}{N})$
random variable $\tilde C_i(k)$ for each $k \in [N]$. 
Then, 
in the exploration process we connect the $i$th vertex 
that was infected or updated while infected to any vertex $k \in [N]$ such that $\tilde C_i(k) =1$. 
This realizes the out-degree $b_i \sim \operatorname{Bin}(N - |R_t|,\nicefrac{\beta}{N})$,
as long as all the new labels are distinct from each other and any previously used
labels. %, the labels in the exploration  and the CPEF agree and $b_i = p_i$.
Moreover, $\p(b_i \neq p_i) \leq \nicefrac{\beta^2}{N}$.
\smallskip

\pagebreak[3]

We can exactly couple the infection and the underlying graph structure up 
to the first time that $B^2$ vertices have been infected or updated while infected, 
as long as the 
following 
%It remains to control the probability of these
 two conditions hold: 
firstly that $p_i = b_i$ for all $i \leq B^2$ and 
% that at any reveal of a vertex we can couple the modified offspring distributions $b_i$ above to the independent $p_i \sim \operatorname{Pois}(\beta)$ of the CPEF; 
secondly that any new label $u_j$ is different from any previously attached labels.
% we don't break the forest assumption in that every half-edge is given a distinct label.%\smallskip
% Lemma \ref{poisson_tv} guarantees a coupling such that for every $i$ we have $b_i=p_i$ with probability $1-\beta/N$, and using this coupling and the union bound we find
By the construction above %we have that
\[
\mathbb{P}\left(
\exists i \in [B^2]:
%\quad
b_i \neq p_i
\right)
 \leq B^2 \cdot
 \frac{\beta^2 }{N}.
\]
For the  assumption on the labels, we use that on the event that $p_i = b_i$ for all $i \in [B^2]$, we have that that the number of labels is bounded by $1+S$ where
\[
S=
\sum_{i=1}^{B^2} p_i \sim
\operatorname{Pois}\left(\beta B^2 \right).
\]
By Markov's inequality we can bound the probability of a label clash by the expected number of label clashes
\[
\mathbb{P}\left(
\exists i, j \in [1+S]:
i \neq j, 
u_i = u_j
\right)
 \leq
\mathbb{E}\Bigg[
 \binom{1+S}{2}\frac{1}{N}
 \Bigg]
 =
 \frac{\beta^2 B^4+2\beta B^2}{2N}.
\]

Altogether we have
\[
\mathbb{P} \left( |I_\infty^\ssup{N}|\geq B^2 \right)\leq \mathbb{P} \left( \text{CPEF infects more than } B \text{ trees} \right)+\frac{2}{B}+\frac{\beta^2 B^2}{N}+ \frac{\beta^2 B^4+2\beta B^2}{2N} .
\]
% Hence, with high probability in $N$, we have coupled the two infection sets and we know
% $
% |I_\infty|\leq B^2
% $
% with high probability in $B$.
Thus, for any $\epsilon>0$ we can, for $B$ large enough, find an $M \in \mathbb{N}$ such that
\[
\sup_{N \geq M} \mathbb{P} \left( |I_\infty^\ssup{N}|\geq B^2 \right)\leq \epsilon
\]
which implies tightness.
%i.e. the set $[B^2]$ contains $1-\epsilon$ mass for a tail of the sequence $(|I_\infty|)_{N \in \mathbb{N}^+}$.
\end{proof}

{\bf Acknowledgements.}
JF was supported by the SAMBa centre for doctoral training at the University of Bath under the EPSRC project EP/L015684/1, then by the \emph{Unit\'e de math\'ematiques pures et appliqu\'es} of ENS Lyon, and finally by NKFI grant KKP 137490.

\DeclareRobustCommand{\VAN}[3]{#3}

\setlength\bibitemsep{0.5\itemsep}
\printbibliography

\end{document}